\documentclass[12pt]{amsart}

\usepackage[all]{xy}
\usepackage{amscd,amsmath,amsthm,amssymb}
\usepackage{mathtools}
\usepackage{mathrsfs}
\usepackage{tikz}
\usepackage{dsfont}
\usepackage{stmaryrd}
\usepackage{enumitem}
\usepackage{underscore}

\definecolor{cadmiumgreen}{rgb}{0.0, 0.42, 0.24}
\usepackage[
colorlinks, citecolor=cadmiumgreen,
pdfauthor={Robin de Jong, Nicole Looper, Farbod Shokrieh}, 
pdfstartview ={FitV},
]{hyperref}
\hypersetup{
	pdftitle={Integral torsion points on abelian varieties over function fields},
}

\usepackage[
	msc-links,
	nobysame,
	lite,
]{amsrefs} 
	
	\usepackage[normalem]{ulem}

	\usepackage[left=3.5cm,top=3.5cm,right=3.5cm]{geometry}
\usepackage{microtype}

\theoremstyle{plain}
\newtheorem{thm}{Theorem}[section]

\newtheorem*{mainthm}{Main Theorem}
\newcommand{\mainthmlabel}[1]{
  \phantomsection
  \hypertarget{#1}{}
}

\newcommand{\mainthmrefcolor}{\hyperlink{thm:main}{\textcolor{black}{Main Theorem}}}

\newtheorem{prop}[thm]{Proposition}
\newtheorem{cor}[thm]{Corollary}
\newtheorem{lem}[thm]{Lemma}

\theoremstyle{remark}
\newtheorem{remark}[thm]{Remark}
\newtheorem{example}[thm]{Example}

\theoremstyle{definition}
\newtheorem{definition}[thm]{Definition}

	\makeatletter
	\renewcommand*\env@matrix[1][*\c@MaxMatrixCols c]{
\hskip -\arraycolsep
\let\@ifnextchar\new@ifnextchar
\array{#1}}
\newcommand*\isom{\xrightarrow{\sim}}
\newcommand{\divisor}{\operatorname{div}}
\newcommand{\ord}{\operatorname{ord}}

\newcommand{\Ker}{\operatorname{Ker}}

\newcommand{\Spec}{\operatorname{Spec}}

\newcommand{\Aut}{\operatorname{Aut}}

\newcommand{\End}{\operatorname{End}}

\newcommand{\Tr}{\operatorname{Tr}}

\newcommand{\Gal}{\operatorname{Gal}}
\newcommand{\Orb}{\operatorname{Orb}}
\newcommand{\Stab}{\operatorname{Stab}}

\def\opn#1#2{\def#1{\operatorname{#2}}}
\opn\Vor{Vor}

\def\rr{\mathbb{R}}
\def\nn{\mathbb{N}}

\def\zz{\mathbb{Z}}
\def\cc{\mathbb{C}}
\def\A{\mathcal{A}}

\def\N{\mathcal{N}}

\def\d{\mathrm{d}}

\def\id{\mathrm{id}}
\def\cl{\mathrm{cl}}

\def\sep{\mathrm{sep}}

\def\an{\mathrm{an}}

\def\tame{\mathrm{tr}}

\def\unr{\mathrm{ur}}

\def\Frob{\mathrm{Frob}}
\def\tor{\mathrm{tor}}
\def\pr{\mathrm{pr}}

\def\eps{\epsilon}
\def\a{a}

\def\tr{\mathrm{tr}}

\def\h{\mathrm{h}}

\numberwithin{equation}{section}

\subjclass[2020]{\href{https://mathscinet.ams.org/msc/msc2020.html?t=11G10}{11G10},
\href{https://mathscinet.ams.org/msc/msc2020.html?t=11G35}{11G35},
\href{https://mathscinet.ams.org/msc/msc2020.html?t=11G50}{11G50},
\href{https://mathscinet.ams.org/msc/msc2020.html?t=14G40}{14G40},
\href{https://mathscinet.ams.org/msc/msc2020.html?t=14K12}{14K12}}

\date{\today}

\begin{document}

\title[Integral torsion points on abelian varieties over function fields]{Integral torsion points on abelian varieties over function fields}

\author{Robin de Jong}
\address{Leiden University \\ Einsteinweg 55  \\ 2333 CC Leiden  \\ The Netherlands }
\email{\href{mailto:rdejong@math.leidenuniv.nl}{rdejong@math.leidenuniv.nl}}

\author{Nicole Looper}
\address{Brown University \\ 
151 Thayer Street  \\ Providence, RI, USA }
\email{\href{mailto:nicole_looper@brown.edu}{nicole_looper@brown.edu}}

\author{Farbod Shokrieh}
\address{University of Washington \\ Box 354350 \\ Seattle, WA 98195 \\ USA }
\email{\href{mailto:farbod@uw.edu}{farbod@uw.edu}}

\begin{abstract} 
We prove an analogue, over global function fields, of a conjecture due to Su-Ion Ih concerning the non-Zariski density of torsion points on abelian varieties that are integral with respect to a given non-special divisor. 
Along the way, we establish a Tate--Voloch type theorem for abelian varieties over completions of global function fields, which allows us to obtain a logarithmic equidistribution result for Galois orbits of torsion points.
\end{abstract}

\maketitle

\setcounter{tocdepth}{1}
\tableofcontents

\thispagestyle{empty}

\section{Introduction} \label{sec:intro}

\renewcommand*{\thethm}{\Alph{thm}}
A \emph{special subvariety} of an abelian variety defined over an algebraically closed field of characteristic zero is defined to be a translate of an abelian subvariety by a torsion point. 
Su-Ion Ih has made the following conjecture, see \cite[\S3]{bir}:

\vspace{2mm}
\noindent \textbf{Ih's conjecture.} \emph{Let $K$ be a number field with algebraic closure $\overline{K}$.
Let $A$ be an abelian variety over~$K$.
Let $D$ be a nonzero effective divisor on $A_{\overline{K}}$.
Assume that at least one irreducible component of $D$ is not special.
Let $S$ be a finite set of places of $K$ including all the archimedean ones.
Let $\A$ be a model of $A$ over the ring of integers of $K$.
Then, the set of torsion points of $A(\overline{K})$ that are $S$-integral with respect to $D$ and $\A$ is not Zariski dense in $A_{\overline{K}}$.}
\vspace{2mm}

As explained in \cite{bir}, one motivation to formulate this conjecture was to have a statement that would compare to the Manin--Mumford conjecture (proved by Raynaud) in a way similar to how Lang's conjecture (proved by Faltings) compares to the Mordell--Lang conjecture (also proved by Faltings).

Ih's conjecture, in general, is still open. In \cite[Theorem~0.2]{bir} one finds a proof of the conjecture in the case that $A$ is an elliptic curve. 
Also in \emph{loc.\ cit.}\ one finds examples showing that the hypothesis that at least one irreducible component of $D$ is special cannot be left out, and that the analogous statement with ``small points'' instead of torsion points is, in general, false.

\subsection{Our main result}

We prove a function field analogue of Ih's conjecture. We note that function field analogues of the Manin--Mumford conjecture, Lang's conjecture, and the Mordell--Lang conjecture, are known by the work of various authors, including Buium \cite{BuiumML}, Hrushovski \cites{hrushovski_ml}, Pillay \cite{Pillay}, Pink--R\"ossler \cites{pr-psi}, Scanlon \cites{scanlon_pos_MM}, and Voloch \cite{voloch}.

A first observation is that the notion of ``special subvariety'' has to be adjusted, see \cites{yamaki_bog, yamaki} and \cite{xie-yuan}. In short, one has to allow for the possibility of constant subvarieties, see Definition~\ref{def:special}. 

\begin{mainthm}\mainthmlabel{thm:main}
Let $K$ be a field of transcendence degree one over a finite field~$k$ of characteristic~$p$. 
Let $\overline{K}$ be an algebraic closure of $K$ and let $\overline{k} \subseteq \overline{K}$ denote the algebraic closure of $k$ in $\overline{K}$.
Let $A$ be an abelian variety over $K$.
Let $D$ be a nonzero effective divisor on $A_{\overline{K}}$.
Assume that at least one irreducible component of $D$ is not $\overline{K}/\overline{k}$-special.
Let $S$ be a finite set of places of $K$ and
let $\A$ be a model of $A$. 
Then, the set of torsion points of $A(\overline{K})$ of prime-to-$p$ order that are $S$-integral with respect to $D$ and $\A$ is not Zariski dense in $A_{\overline{K}}$.
\end{mainthm}
We mention that Petsche in \cite{petsche_equid} has proven the analogue of \cite[Theorem~0.2]{bir} for elliptic curves over a function field over any field. Petsche's result is quantitative, in the sense that it gives an explicit bound for the number of torsion points that are $S$-integral with respect to a given non-torsion point and a given Weierstrass model.

We remark that the set of prime-to-$p$ torsion points of $A({\overline{K}})$ is Zariski dense. The focus on prime-to-$p$ torsion points allows us to leverage Galois-theoretic methods and avoid inseparability issues.

\subsection{Ingredients of the proof}
The strategy of our proof of the ~\mainthmrefcolor~ is similar to the idea in \cite{bir}. Roughly speaking, if $D$ is an ample integral divisor on $A_{\overline{K}}$ and there exists a generic net of $S$-integral prime-to-$p$ order torsion points with respect to $D$, then this forces $D$ to have vanishing canonical height. This in turn implies that $D$ is $\overline{K}/\overline{k}$-special by the geometric Bogomolov conjecture for divisors on abelian varieties, as shown by Yamaki \cite{yamaki}.

Besides Yamaki's result, we use a Galois equidistribution result for torsion points on abelian varieties over function fields. Such a result is proved by Faber \cite{faber_equid} and Gubler \cite{gub_equid}. But, in fact, we need a more refined version of their result, namely a certain result of ``logarithmic equidistribution'' type (see Theorem~\ref{thm:log_equid}). The need for logarithmic test functions comes from our need to use local height functions.

A crucial tool in our proof is the use of a certain normalization of local canonical heights on abelian varieties that was introduced and studied by the first and third named authors in \cite{djs_canonical}, using the framework of Berkovich analytic spaces. This normalization generalizes the well-known normalization of local canonical heights on elliptic curves given by Tate. Importantly, the normalized local canonical height vanishes for integral points at places of good reduction (see Corollary~\ref{cor:normalized_good}). We further have the global formula \eqref{eq:local_global_bis} for the canonical height, which yields a key result (Proposition \ref{prop:local_to_global}) in the proof of the ~\mainthmrefcolor.

\subsection{A Tate--Voloch type result} Our result of logarithmic equidistribution type is obtained by first proving a suitable analogue of the so-called Tate--Voloch conjecture for abelian varieties defined over local fields in characteristic~$p$ (see Theorem~\ref{thm:tate-voloch-app}). The Tate--Voloch conjecture (\cite[\S3]{TateVoloch}) states that for any semi-abelian variety $G$ over $\mathbb{C}_p$, where $\mathbb{C}_p$ is the completion of $\overline{\mathbb{Q}_p}$ for some prime $p$, and for any closed subvariety $X$ of $G$, there exists an $\epsilon>0$ such that any torsion point of $G$ either lies in $X$ or lies at distance at least $\epsilon$ from $X$. See \S\ref{sec:distance} for details on how this distance is defined. This conjecture was proved by Scanlon \cites{Scanlon, scanlon_crelle} using model-theoretic techniques, with the assumption that $G$ must be defined over $\overline{\mathbb{Q}_p}$. Corpet \cite{corpet_thesis} later found a different proof that avoided model theory, at the expense of the further assumption that both $G$ and $X$ are defined over $\overline{\mathbb{Q}_p}$. Corpet's proof translates the analytic proximity of torsion points to $X$ into an algebro-geometric condition on points valued in an ultrafield, thereby recasting the argument in purely field-theoretic and valuation-theoretic terms. Our proof of our Tate--Voloch type result for abelian varieties over completions of global function fields follows Corpet's proof to a large extent. We adapt and simplify Corpet's argument, among other ways, by utilizing the fact that the ``satellite variety" defined in \eqref{eqn:Z} is special.

\subsection{Back to number fields} 
Our proof of the ~\mainthmrefcolor~ provides a template for proving Ih's conjecture over number fields. The necessary analogue in the number field setting of the Bogomolov conjecture is provided by~\cite[Corollary~3]{zhang_equid}, and the necessary analogue of the Galois equidistribution result due to Faber and  Gubler is provided by~\cite[Theorem~1.1]{zhang_equid}. 

In order to derive from this the necessary logarithmic equidistribution result at the non-archimedean places of the number field, one may invoke the Tate--Voloch conjecture for abelian varieties defined over local fields in characteristic zero as shown by Scanlon.  The ``only'' item that is missing from this template in the number field setting is a suitable logarithmic equidistribution result for Galois orbits of torsion points at archimedean places.

\subsection{Structure of the paper} After a short section discussing preliminaries, our paper will consist of two parts.  Part~\ref{part_one} (\S\S\ref{appendix:invariant}--\ref{sec:appendixB}) is aimed at proving the Tate--Voloch type result that is needed in our proof of the ~\mainthmrefcolor. Part~\ref{part_two} (\S\S\ref{sec:normalized}--\ref{sec:proof_last_reduction_step}) is devoted to the proof of the ~\mainthmrefcolor. We split the paper into these two parts because the techniques and tools appearing in both parts are rather different in nature. 

\subsection{Acknowledgements} The authors would like to thank Ariyan Javanpeykar, Damian R\"ossler, Tom Scanlon and Junyi Xie for helpful comments and correspondences around this project. The second named author was funded by NSF DMS-2302586 as well as a Sloan Research Fellowship. The third named author was partially supported by NSF CAREER DMS-2044564. 

\renewcommand*{\thethm}{\arabic{section}.\arabic{thm}}

\section{Preliminaries} \label{sec:stab_spec}

Throughout, a \emph{variety} over a field $\mathds{k}$ will be a reduced separated $\mathds{k}$-scheme of finite type.

Let $\mathds{K}$ be an algebraically closed field, let $A$ be an abelian variety over $\mathds{K}$, and let $X$ be a closed subvariety of $A$.  

\begin{definition}
Let $P$ be the closed subgroup scheme of $A$ characterized by the following property: for all $\mathds{K}$-schemes $S$ and all $S$-valued points $a \colon S \to A$, translation by $a$ on the product $A \times S$ maps $X \times S$ into itself if and only if $a$ factors through $P$. We define $\Stab_A(X)$ to be the reduced closed subgroup scheme $P_{\mathrm{red}} \subseteq A$, and call it the \emph{stabilizer} of $X$ in $A$.
\end{definition}

The stabilizer exists by \cite[Expos\'e VIII, Ex.\ 6.5(e)]{SGA3}.  The identity component of the stabilizer of $X$ in $A$ is an abelian subvariety of $A$.
 \begin{lem} \label{lem:stab_quotient_pre}
Let $X \subseteq A$ be a closed irreducible subvariety and let $H$ be a closed subgroup variety of $A$. Then, we have a canonical identification
\[ \Stab_{A/H} (X/H) = \Stab_A(X)/H . 
\]
\end{lem}
\begin{proof} This follows from \cite[Proposition~4.1]{pr-psi}. 
\end{proof}

Let $\mathds{k}$ be an algebraically closed subfield of $\mathds{K}$.
\begin{definition}
The \emph{$\mathds{K}/\mathds{k}$-trace} of $A$ is a final object in the category of pairs $(B, f)$ consisting of an abelian variety $B/\mathds{k}$ and a morphism of abelian varieties $f \colon B \otimes_{\mathds{k}} \mathds{K} \to A$. It is denoted by $(A^{\mathds{K}/\mathds{k}},\Tr^{\mathds{K}/\mathds{k}}_A)$.
\end{definition}

The $\mathds{K}/\mathds{k}$-trace of $A$ exists by \cite[Theorem~6.2]{conrad_chow}. 

\begin{definition}\label{def:special}
 An irreducible closed subvariety $X$ of $A$ is called $\mathds{K}/\mathds{k}$-\emph{special} if there is an abelian subvariety $G$ of $A$, a torsion point $\tau \in A(\mathds{K})$, and a closed subvariety $X'$ of $A^{\mathds{K}/\mathds{k}}$ such that $X = G + \Tr^{\mathds{K}/\mathds{k}}_A(X' \otimes_\mathds{k} \mathds{K}) + \tau$. 
\end{definition}
 \begin{lem} \label{lem:stab_quotient}
Let $X$ be an irreducible closed subvariety of $A$. Let $H$ be a closed algebraic subgroup of $\Stab_A(X)$. Then, $X$ is a $\mathds{K}/\mathds{k}$-special subvariety of~$A$ if and only if $X/H$ is a $\mathds{K}/\mathds{k}$-special subvariety of~$A/H$. 
\end{lem}
\begin{proof} 
See \cite[Lemma~7.1]{yamaki_strict}.
\end{proof}

\part{Proximity of torsion over function fields} \label{part_one}

The aim of this part is to state and prove a suitable result of Tate--Voloch type for abelian varieties over local fields in positive characteristics (see Theorem~\ref{thm:tate-voloch-app}). We closely follow the ideas of Chapter~2 of the PhD thesis of Corpet~\cite{corpet_thesis}, but we need adjustments in several places. 

In \S\ref{appendix:invariant} we review a result due to Pink--R\"ossler about invariant subvarieties of abelian varieties, and deduce a consequence.
In \S\ref{sec:appendixA} we review the notion of distance functions on varieties over fields with a non-archimedean absolute value, following Voloch's exposition in \cite{voloch_distance}.
In \S\ref{sec:appendix_discrete} we show that torsion is discrete on abelian varieties defined over local fields.
In \S\ref{appendix:galois} we prove some Galois-theoretic results on sequences approximating subvarieties of abelian varieties.
In \S\ref{sec:appendixB} we finally state and prove our Tate--Voloch type result in positive characteristic.

\section{Invariant subvarieties of abelian varieties} \label{appendix:invariant}

Let $\mathds{K}$ be an algebraically closed field. Let $A$ be an abelian variety over $\mathds{K}$ and let $\varphi \colon A \to A$ be an isogeny.  

\begin{definition} (See \cite[Definition~2.1]{pr-psi})
\begin{itemize}
\item[(i)] We call $A$ \emph{pure of weight $0$} if $\varphi$ is an automorphism of finite order of $A$.
\item[(ii)] We call $A$ \emph{pure of weight $\alpha=r/s$} for positive integers $r, s$ if $\varphi^s = \Frob_{p^r}$ for some model of $A$ over $\mathbb{F}_{p^r}$.
\end{itemize}
\end{definition}

\subsection{A result of Pink and R\"ossler}

In this section we work with  an irreducible closed subvariety $X \subseteq A$ with the property that $\varphi(X)=X$. Let $B$ denote the identity component of the stabilizer $\Stab_A(X)$. Since $\varphi(X)=X$, the homomorphism $\varphi$ maps $\Stab_A(X)$ into itself; as $B$ is the identity component of $\Stab_A(X)$, its image $\varphi(B)$ is a connected subgroup of $\Stab_A(X)$ and is thus contained in $B$. Then, $B$ is an abelian subvariety of $A$ satisfying  $\varphi(B)\subseteq B$, so that $\varphi$ induces an endomorphism of $X/B$.

The following is a special case of a result due to Pink and R\"ossler, see \cite[Theorem~3.1]{pr-psi}. (For the ``moreover'' statement in the last sentence, see \cite[Remark~3.2]{pr-psi}). 
\begin{thm} \label{thm:psi-pre} 
There exist finitely many $\varphi$-invariant abelian subvarieties $A_\alpha$ of $A$ pure of weight $\alpha \ge 0$, closed irreducible subvarieties $X_\alpha \subseteq A_\alpha$, elements $a_\alpha \in A_\alpha$ and $a \in A$ such that the composition of the summation map and the projection map
\[ h \colon \prod_\alpha A_\alpha \rightarrow A \rightarrow A/B \]
has finite kernel, and one has $\varphi(X_\alpha) = X_\alpha + a_\alpha$ and
\[ X/B = \overline{a} + h\left( \prod_\alpha X_\alpha \right).\] Moreover, for all $\alpha>0$ the $X_\alpha$ may be assumed to be defined over a finite field 
and to satisfy the equality $\varphi(X_\alpha)=X_\alpha$.
\end{thm}
We use Theorem~\ref{thm:psi-pre} to deduce the following result. 
\begin{thm} \label{thm:dichotomy} Let $\mathds{k}$ be the algebraic closure of the prime field in $\mathds{K}$. \begin{itemize}
\item[(i)] Assume that $A$ has no positive dimensional abelian subvariety that is pure of weight $0$. Then, $X$ is a $\mathds{K}/\mathds{k}$-special subvariety of $A$.
\item[(ii)] Assume that the endomorphism $\varphi - \id_A$ of $A$ is nilpotent. Then, $\varphi$ acts as an automorphism of finite order on $X/B$.
\end{itemize}
\end{thm}
\begin{proof} (i) By Lemma~\ref{lem:stab_quotient} it suffices to show that $X/B$ is a $\mathds{K}/\mathds{k}$-special subvariety of $A/B$. By Theorem~\ref{thm:psi-pre} there exist finitely many $\varphi$-invariant abelian subvarieties $A_\alpha$ of $A$ and closed irreducible subvarieties $X_\alpha$ of $A_\alpha$, all defined over a finite field and satisfying $\varphi(X_\alpha)=X_\alpha$, as well as a point $a \in A$,
such that  
\begin{equation} \label{eqn:define_a}
X/B = \overline{a} + h \left( \prod_\alpha X_\alpha \right) \, .
\end{equation}
Here  $h \colon \prod_{\alpha} A_\alpha \rightarrow A \rightarrow A/B $ is the composition of the summation map and the projection map.  Since $h\!\left(\prod_{\alpha>0} X_\alpha\right)$ is a closed subvariety of $A/B$ defined over
a finite field, hence over $\mathds{k}$, the expression
\eqref{eqn:define_a} shows that $X/B$ is a translate of a $\mathds{k}$-subvariety of $A/B$. 
Therefore, in order to conclude that $X/B$ (and hence $X$) is $\mathds{K}/\mathds{k}$-special, it suffices to
prove that $\overline a\in A/B$ is a torsion point.
 From \eqref{eqn:define_a} we deduce, first of all, that $\varphi(\overline{a})-\overline{a}$ lies in $\Stab_{A/B}(X/B)$. The latter is identified with $\Stab_A(X)/B$ by Lemma~\ref{lem:stab_quotient_pre} and is therefore a finite subgroup of $A/B$. Let $N\in\zz_{>0}$ be such that
$N\bigl(\varphi(\overline a)-\overline a\bigr)=0$ in $A/B$.
Then
\[0=N\bigl(\varphi(\overline a)-\overline a\bigr)
 =\varphi(N\overline a)-N\overline a,\] so $N\overline a\in\Ker(\varphi-\id_{A/B})$. We claim that $\Ker(\varphi-\id_{A/B})$ is finite. Indeed, its identity component is an abelian subvariety of $A/B$ on which
$\varphi$ acts as the identity, hence it is pure of weight $0$. The inverse image of this abelian subvariety in $A$ is therefore a
$\varphi$-invariant abelian subvariety of $A$ of pure weight $0$,
which must be trivial by assumption.
Hence $\Ker(\varphi-\id_{A/B})$ is finite, so $N\overline a$ is a torsion
point of $A/B$, and therefore $\overline a$ is torsion as well.

(ii) Let $H$ be a $\varphi$-invariant abelian subvariety of $A$ defined over a finite field on which a power of $\varphi$ acts as a power of the Frobenius endomorphism. We want to show that $H=0$. For every $r\in\zz_{>0}$ we have \[\varphi^r-\id_A=(\varphi-\id_A)(\varphi^{r-1}+\cdots+\id_A),\] and the second factor commutes with $\varphi-\id_A$; hence $\varphi-\id_A$ nilpotent implies that $\varphi^r-\id_A$ is nilpotent as well. Therefore, for any prime $\ell\neq p$, all eigenvalues of $\varphi^r$ on $T_\ell(A)$ are equal to $1$, and in particular the same holds for the restriction of $\varphi^r$ to the submodule $T_\ell(H)\subseteq T_\ell(A)$. Choose $r$ such that $\varphi^r|_H$ is a power of the Frobenius endomorphism (as $H$ is defined over a finite field). Then, the characteristic polynomial of $\varphi^r$ on $T_\ell(H)$ is a Weil polynomial, whose roots have absolute value $>1$ whenever $H\neq 0$. This forces $H=0$.

Consequently, in the decomposition given by Theorem~\ref{thm:psi-pre}, all factors $A_\alpha$ with $\alpha>0$ are trivial.  
By Theorem~\ref{thm:psi-pre} there exist $N_1\in\zz_{>0}$ and a $\varphi$-invariant abelian subvariety $A_0$ of $A$ such that $\varphi^{N_1}$ acts as the identity on $A_0$, a closed irreducible subvariety $X_0$ of $A_0$, and $a\in A$ such that
\[
X/B=\overline{a}+h(X_0).
\]
Here $h\colon A_0\to A\to A/B$ is the composition of the inclusion map and the projection map.  
As in (i), we deduce that $\varphi(\overline{a})-\overline{a}$ lies in the stabilizer of $X/B$, which is a finite subgroup and moreover is $\varphi$-invariant. It follows that there exists $N_2\in\zz_{>0}$ such that $\varphi^{N_2}(\overline{a})=\overline{a}$.  
Now write $N=N_1N_2$. Then, $\varphi^N$ acts as the identity on both $\overline{a}$ and $h(X_0)$, and hence on $X/B$.\end{proof}

\subsection{A construction of invariant subvarieties}  \label{sec:construction_inv}

In this section we assume we are given a monic polynomial $G \in \zz[T]$ of degree~$d$ with nonzero constant term. Write $d$ for the degree of $G$ and let $X \subseteq A$ be any closed irreducible subvariety.  In the sequel we will consider a certain invariant closed subvariety of $A^d$ constructed from $G$ and $X$ that we would like to call the ``satellite variety'' of the pair $(G,X)$. We refer to ~\cite[\S2.2]{corpet_thesis} and \cite[\S4]{roessler_MM_ML} for further details of this construction, which runs as follows.

Write $G=T^d-\sum_{i=0}^{d-1}a_iT^i$, and let \[M=\left(\begin{matrix}0&1&0&0\\\vdots&0&\ddots&\vdots\\0&\cdots&\cdots&1\\a_0&a_1&\cdots&a_{d-1}\end{matrix}\right)\] be the companion matrix of $G$.  We denote by $\psi \colon A^d \to A^d$  the isogeny defined by the matrix $M$ and consider the (possibly empty) closed subvariety
\begin{equation}\label{eqn:Z}
Z=\bigcap_{m\ge0}\psi_*^m\left(\bigcap_{r\ge0}\psi^{*,r}(X^d)\right)\subseteq X^d . 
\end{equation} 
Here, $\psi^*$ (resp.\ $\psi_*$) denote the operations of pullback (resp.\ proper pushforward) along $\psi$ for closed subvarieties of $A^d$. As is noted in~\cite[\S4]{roessler_MM_ML}, we have the set-theoretic equality $\psi(Z)=Z$.  Since $\psi$ permutes the irreducible components of $Z$, there is thus a positive integer $m$ such that $\psi^m$ stabilizes each of the irreducible components of $Z$. 

\section{Distance functions and ultrafields}\label{sec:distance} \label{sec:appendixA}

Let $F$ be a field, complete with respect to a non-trivial non-archimedean absolute value $|\cdot|$.
We recall from \cite{voloch_distance} the definition and main properties of distance functions (we refer to \cite{silverman_distance} for an equivalent treatment).  

Let $\mathcal{O}=\{x\in F:|x|\le 1\}$ be the ring of integers of $F$, and for any $0<\epsilon \le 1$, let $\mathcal{M}_{\epsilon}=\{x\in F:|x|<\epsilon\}$. This is an ideal of $\mathcal{O}$, and we write $\mathcal{O}_{\epsilon}=\mathcal{O}/\mathcal{M}_\epsilon$. When $\mathcal{X}$ is a finite type scheme over $\mathcal{O}$, we write $\mathcal{X}_\epsilon$ for the $\mathcal{O}_\epsilon$-scheme $\mathcal{X}\otimes_{\mathcal{O}}\mathcal{O}_\epsilon$. 

One has a natural non-archimedean analytic topology on $\mathcal{X}(\mathcal{O})$ defined as follows: for each affine open $U=\mathrm{Spec}(A)\subseteq \mathcal{X}$, where $A=\mathcal{O}[x_1,\dots,x_n]/I$, the set of $\mathcal{O}$-points $U(\mathcal{O})$, which is identified with the set \[\{(a_1,\dots,a_n)\in\mathcal{O}^n : f(a_1,\dots,a_n)=0 \text{ for all } f\in I\},\] is given the subspace topology induced from $\mathcal{O}^n$. The topology on $\mathcal{X}(\mathcal{O})$ is defined by gluing the topologies on these affine patches. 

\begin{definition}\label{def:scheme_distance}
Let $\mathcal{X}$ be a scheme of finite type over $\mathrm{Spec}(\mathcal{O})$. We define a distance function $d$ of elements $P\in \mathcal{X}(\mathcal{O})$ to closed subschemes $\mathcal{Y}\subseteq \mathcal{X}$ as follows. Let \[d(P,\mathcal{Y})=\inf\{\epsilon>0:P_\epsilon\in \mathcal{Y}_\epsilon\},\]
provided this set is non-empty; otherwise set $d(P,\mathcal{Y})=1$. Note that $d(P,\mathcal{Y})=0$ if and only if $P \in \mathcal{Y}(\mathcal{O})$. 
\end{definition}

An equivalent definition is as follows: set $d(P,\mathcal{Y})=1$ if $\mathcal{Y}$ and the image of $P$ in $\mathcal{X}$ have empty intersection. If $\mathcal{Y}$ and the image of $P$ in $\mathcal{X}$ have non-empty proper intersection, then the pullback \[P^*\mathcal{Y}\subseteq \Spec(\mathcal{O})\] is a closed subscheme. Since the generic point of $\Spec(\mathcal{O})$ maps into $\mathcal{Y}$ if and only if $P$ factors through $\mathcal{Y}$, the assumption that the intersection is proper implies that this closed subscheme is supported on the closed point of $\Spec(\mathcal{O})$, and hence is defined by an ideal $I_{P,\mathcal{Y}}\subseteq\mathcal{O}$. Then, $d(P,\mathcal{Y}) = \inf \{\epsilon>0: I_{P,\mathcal{Y}} \subseteq\mathcal{M}_\epsilon\}$.

\begin{prop}\label{voloch_distance} 
Let $\mathcal{X}$ be a scheme of finite type over $\mathrm{Spec}(\mathcal{O})$. The distance function defined above has the following properties: 
\begin{enumerate}[label=(\roman*)]  
\item\label{item:i} If $\mathcal{Y}$ is an effective Cartier divisor on $\mathcal{X}$ and $U$ is an open subset of $\mathcal{X}$ on which $\mathcal{Y}$ is given by the principal ideal $(z)$ for $z\in\mathcal{O}_{\mathcal{X}}(U)$, then $d(P,\mathcal{Y})=|z(P)|$ for $P\in U(\mathcal{O})$.  

\item\label{item:ii} For closed subschemes $\mathcal{Y}$ and $\mathcal{Z}$ of $\mathcal{X}$ and any point $P\in \mathcal{X}(\mathcal{O})$, \[d(P,\mathcal{Y}\cap\mathcal{Z})=\max\{d(P,\mathcal{Y}),d(P,\mathcal{Z})\}.\]  

\item\label{item:iii} For closed subschemes $\mathcal{Y}$ and $\mathcal{Z}$ of $\mathcal{X}$ and any point $P\in \mathcal{X}(\mathcal{O})$, \[d(P,\mathcal{Y}\cup\mathcal{Z})=\min\{d(P,\mathcal{Y}),d(P,\mathcal{Z})\}.\]  

\item\label{item:iv} Let $\mathcal{Y}$ be a scheme of finite type over $\mathrm{Spec}(\mathcal{O})$, and let $f \colon \mathcal{X}\to \mathcal{Y}$ be a morphism of $\mathcal{O}$-schemes. For any closed subscheme $\mathcal{Z}$ of $\mathcal{Y}$ and any $P\in \mathcal{X}(\mathcal{O})$, we have
 \[d(P,f^{*}(\mathcal{Z}))=d(f(P),\mathcal{Z}).\]  
Here we write $f^*(\mathcal{Z}) = \mathcal{Z} \times_{\mathcal{Y}} \mathcal{X}$. In particular when $f$ is proper we have
\[ d(P,\mathcal{Z}) \ge d(f(P),f(\mathcal{Z})).\]

\item\label{item:v} As a function on $\mathcal{X}(\mathcal{O})\times \mathcal{X}(\mathcal{O})$, $d(P,Q)$ defines a metric which induces the analytic topology on $\mathcal{X}(\mathcal{O})$. 
It satisfies the ultrametric inequality \[d(P,R)\le\max\{d(P,Q),d(Q,R)\} \, , \quad P,Q,R\in \mathcal{X}(\mathcal{O}),\]
with equality holding if $d(P,Q),d(Q,R)$ are distinct.  

\item\label{item:vi} When  $E/F$ is a Galois extension with ring of integers $\mathcal{O}'$, $\sigma\in\Gal(E/F)$ and $\mathcal{Y}$ is a closed subscheme of $\mathcal{X}\otimes_{\mathcal{O}}\mathcal{O}'$, then we have \[d(P^{\sigma},\mathcal{Y}^{\sigma})=d(P,\mathcal{Y}) \, , \quad P\in \mathcal{X}(\mathcal{O}'). \]
\end{enumerate} 
\end{prop}
\begin{proof}
All items except (iii) are given in \cite[Theorem~1]{voloch_distance}. Item (iii) follows immediately from the definitions.
\end{proof} 
\begin{remark} \label{rem:translation_inv}
By property (iv) of Proposition~\ref{voloch_distance}, the distance is $\Aut(\mathcal{X}/\mathcal{O})$-invariant. In particular, when $\mathcal{X}$ is a group scheme of finite type over $\mathcal{O}$, the distance is translation-invariant. See also \cite[Corollary~2]{voloch_distance}.
\end{remark}
Now let $X$ be a projective variety over $F$ and let $Z \subseteq X$ be a closed subvariety. We obtain distance functions $d(\cdot,Z)$ on $X(F)$ by taking models $\mathcal{X}$ of $X$ over $\mathcal{O}$, and setting $d(P,Z) = d(P,\mathcal{Z})$ where $\mathcal{Z} \subseteq \mathcal{X}$ is any model of $Z$ over $\mathcal{O}$ realized as a closed subscheme of $\mathcal{X}$. 

Let $\mathbb{F}$ be the completion of an algebraic closure of $F$. Then, any distance function $d(\cdot,Z)$ thus obtained for closed subvarieties $Z \subseteq X$ naturally extends to a function on $X(\mathbb{F})$. Further, as is noted in \cite{voloch_distance}, the notion of distance functions to closed subvarieties as above is equivalent with the notion of \emph{local height functions} with respect to closed subvarieties as developed in \cite[\S\S1--2]{silverman_distance}. This connection implies that for any two distance functions $d_1(\cdot,Z), d_2(\cdot,Z)$ as above for a given closed subvariety $Z  \subseteq X$, there exist constants $c_1,c_2>0$ with the property that 
\[ \textup{ for all }P \in X(\mathbb{F}) \quad : \quad c_1d_1(P,Z)\le d_2(P,Z)\le c_2d_1(P,Z).\]
In fact, one obtains the local height functions $\lambda_Z$ from \cite{silverman_distance} by setting
\[ \lambda_Z(P) = - \log d(P,Z)  \quad \textup{ for all } P \in (X \setminus Z) (\mathbb{F}) . \]

\begin{definition} \label{def:conv_bad} Let $Z \subseteq X$ be a closed subvariety.
We say that a sequence $(P_n)_{n \in \mathbb{N}}$ in $X(\mathbb{F})$ \emph{converges to} $Z$ if $d(P_n,Z) \to 0$ for $n \to \infty$. By the above inequalities the notion of convergence does not depend on the choice of distance function $d(\cdot,Z)$ for $Z$. 
\end{definition}

\begin{definition} \label{def:bad} Let $Z \subseteq X$ be a closed subvariety and let $\Sigma \subseteq X(\mathbb{F})$ be a subset.   We say that $Z$  is \emph{unapproximable} by $\Sigma$ if given any distance function $d(\cdot,Z)$ with respect to~$Z$ there exists $\eps > 0$ such that for all $P \in \Sigma$ we either have $d(P,Z) > \eps$ or $P \in Z$. 
\end{definition}
Clearly we have that $Z$ is unapproximable by $\Sigma$ if and only if  for every sequence $(P_n)_n$ of points in $\Sigma$ that converges to $Z$ we have $P_n \in Z$ for $n$ large enough.

\subsection{Ultrafields} \label{constr:D}

Using ultrafields, one may turn sequences converging to a closed subvariety into certain algebraic points of that subvariety. This construction will be useful later on.
\begin{definition} 
Let $R_F$ denote the ring of bounded sequences $(a_i)_{i\in\mathbb{N}}$ in $F$. Let $\mathfrak{U}$ be a non-principal ultrafilter on $\nn$, i.e., an ultrafilter that does not contain any finite set. If $(a_n)_n$ is a sequence and $\mathcal{P}$ is some property, then we say that $a_n$ satisfies $\mathcal{P}$ \emph{for $\mathfrak{U}$-almost $n$} if the set of all $n$ such that $a_n$ satisfies $\mathcal{P}$ belongs to $\mathfrak{U}$.

Associated to $\mathfrak{U}$ one has an ideal $I_\mathfrak{U}$ of $R_F$ defined by the condition
\[(a_i)_{i\in\mathbb{N}}\in I_{\mathfrak{U}}\Longleftrightarrow 
\forall n \in \nn,\hspace{1mm}\{i\in\mathbb{N}:|a_i| \le e^{-n}\}\in\mathfrak{U} . \] 
It turns out that  $I_\mathfrak{U}$ is a maximal ideal of $R_F$. The quotient ring
\[ D = D_\mathfrak{U} = R_F/I_\mathfrak{U} \]
is naturally an overfield of $F$, called the \emph{ultrafield over $F$ with respect to $\mathfrak{U}$}. Every $\sigma \in \Aut(F,|\cdot|)$ naturally induces an element of $\Aut(D)$ that we also denote by~$\sigma$. 
\end{definition}
Let again $X$ be a projective variety over $F$ and let $(P_n)_{n}$ be a sequence in $X(F)$. Then, $(P_n)_n$ naturally gives rise to a point of $X(D)$ that we denote by~$P^*$. When $Z \subseteq X$ is a closed subvariety and the sequence $(P_n)_{n}$ converges to $Z$, we have $P^* \in Z(D)$.

\section{Discreteness of torsion} \label{sec:appendix_discrete}

In this section we work with a non-trivially valued non-archimedean field $(F,|\cdot|)$ having \emph{finite} residue field $k$ of characteristic~$p$.  Let $\mathbb{F}$ be the completion of an algebraic closure $\overline{F}$ of $F$. The absolute value $|\cdot|$ has a unique extension to $\mathbb{F}$ that we again denote by $|\cdot|$.

Let $A$ be an abelian variety over $F$. Our aim here is to show that the torsion subgroup of $A(\mathbb{F})$ is discrete in the non-archimedean analytic topology. 
This is a classical result known as ``Mattuck's theorem'' in the case where $F$ has characteristic zero \cite{mattuck}. The positive characteristic case is probably well-known to experts. In \cite[Lemma~8]{petsche_equid} a proof is given in the case where $A$ is an elliptic curve. The proof from \emph{loc.\ cit.} carries over \emph{mutatis mutandis} to our setting. For convenience we give details below.

\subsection{Kernel of reduction and formal group}

Let $A_1$ denote the kernel of the reduction map, i.e., let $E/F$ be any field extension with $E$ complete. Then we let $A_1(E)$ denote the kernel of the natural map from $A(E)$ to the rational points of the special fiber of the N\'eron model of $A$ over the valuation ring of $E$. 

In distance-theoretic terms the kernel of reduction has a convenient description as follows. Let $\mathcal{M}_E$ be the maximal ideal of the valuation ring of $E$. Let $d(\cdot,O)$ be any distance function on $A(E)$ with respect to the origin $O$ (see \S\ref{sec:distance}). 
By Proposition~\ref{voloch_distance}\,(iv), applied to the canonical morphism from any model of $A$ to its N\'eron model, the inequality $d(P,O)<1$ is independent of the chosen model and is equivalent to $P$ reducing to $O$. 
Thus we have
\[
A_1(E)=\{P\in A(E): d(P,O)<1\}.
\]
This, in particular, shows that $A_1(E)$ is an open subgroup of $A(E)$.

Let $\mathcal{A}$ be the N\'eron model of $A$ over the valuation ring $\mathcal{O}$ of $F$. Let $\widehat{\mathcal{A}}$ be the formal completion of $\mathcal{A}$ along the identity section. Then, $\widehat{\mathcal{A}}$ is a commutative formal Lie group of dimension~$g=\dim A$ over $\mathcal{O}$. A basic reference for this is \cite[Part II, Chapter IV]{serre_lie}.

Choose formal parameters $Z=(Z_1,\dots,Z_g)$ such that the identity section of $\widehat{\mathcal{A}}$ corresponds to $Z=0$ and such that $\widehat{\mathcal{A}} \cong \mathrm{Spf}(\mathcal{O}[[Z_1,\dots,Z_g]])$. This induces an isomorphism of analytic manifolds $\widehat{\mathcal{A}}(\mathcal{M}_E)\cong \mathcal{M}_E^g$ sending the identity to $0$.
 As is explained in for instance \cite[\S2.5]{clark_xarles}, we have an isomorphism of analytic groups $A_1(E) \cong \widehat{\mathcal{A}}(\mathcal{M}_E)$. By transport of structure from $\mathcal{M}_E^g$ we obtain a metric on $A_1(E)$ compatible with the analytic topology on $A_1(E)$.

\begin{lem}\label{lem:kernel_reduction}
The torsion subgroup $A_1(\mathbb{F})_\tor$ of $A_1(\mathbb{F})$  is a $p$-group.
\end{lem}
\begin{proof} Let $E / F$ be a finite field extension. Then, $E$ is a local field and by \cite[Part II, Chapter IV, Section 9, Corollary 2]{serre_lie} we have that $\widehat{\mathcal{A}}(\mathcal{M}_E)$ and hence $A_1(E)$ is a pro-$p$ group. In particular $A_1(E)_\tor$ is a $p$-group, and thus so is $A_1(\overline{F})_\tor$. The proof is finished by observing that all torsion in $A_1(\mathbb{F})$ is defined over $\overline{F}$.
\end{proof}

\begin{lem}\label{thm:TV_origin}
The torsion subgroup $A(\mathbb{F})_{\mathrm{tor}}$ of $A(\mathbb{F})$ is discrete in $A(\mathbb{F})$.
\end{lem}
\begin{proof}

Since $A_1(\mathbb{F})$ is an open subgroup of the topological group $A(\mathbb{F})$, it suffices to show that $A_1(\mathbb{F})_{\mathrm{tor}}$ is discrete in $A_1(\mathbb{F})$. Since translations by elements of $A_1(\mathbb{F})$ are homeomorphisms of $A_1(\mathbb{F})$, it suffices to show that the origin of $A$ has an open neighborhood in $A_1(\mathbb{F})$ containing no non-trivial torsion points.

Let $F(Z,W) \in \mathcal{O}[[Z,W]]^g$ denote the formal group law of $\widehat{\mathcal{A}}$, so that
\[
F(Z,W) = Z + W + G(Z,W),
\]
where $G(Z,W)$ has no constant or linear terms. Recall that for $n \in \zz_{>0}$ the formal multiplication-by-$n$ map of $\widehat{\mathcal{A}}$ is given recursively by
\[
[n](Z) = \underbrace{F(Z, F(Z, \dotsc, F(Z, Z)))}_{n \text{ times}}.
\]
We show by induction that the lowest-degree term of $[n](Z)$ is $nZ$. Indeed one has $[1](Z) = Z$. Now suppose
\[
[n](Z) = nZ + R^{(n)}(Z),
\]
where $R^{(n)}(Z)$ has total degree $\ge 2$. Then,
\[[n+1](Z)=F([n](Z),Z)=[n](Z)+Z+G([n](Z),Z),\]
and $G([n](Z),Z)$ has degree $\ge 2$, so the degree-$1$ term of $[n+1](Z)$ is $(n+1)Z$, as desired. 
\par
In particular, we have
\[ [p](Z)=pZ+\text{(terms of degree $\ge 2$)} . \]
Let $\mathcal{M}$ denote the maximal ideal of $\mathbb{F}$.
As $|p|<1$, we see that multiplication-by-$p$ in the analytic group $\widehat{\mathcal{A}}(\mathcal{M})$ strictly contracts near~$0$.
We conclude that there exists a $\rho>0$ such that the open ball $B(0,\rho)$ around the origin of radius $\rho$ inside $A_1(\mathbb{F})$ satisfies
\[ [p](B(0,\rho))\subsetneq B(0,\rho) . \]
We may moreover assume that $B(0,\rho)$ contains no nontrivial $p$-torsion point of $A$, as $A[p]$ is  a finite group. For such a $\rho$, we must also have that $B(0,\rho)$ does not contain any $p$-power torsion point; for otherwise, if $P \in B(0,\rho)$ were of exact order $p^r$ with $r \ge 2$, then $[p^{r-1}](P) \in B(0,\rho)$ would be a point of exact order $p$, and we arrive at a contradiction. 

On the other hand, by Lemma~\ref{lem:kernel_reduction}, we know that $A_1(\mathbb{F})_{\mathrm{tor}}$ is a $p$-group, and hence every torsion point of $A_1(\mathbb{F})$ has $p$-power order. We conclude that $B(0,\rho)$ contains no nontrivial torsion points, which completes the proof.\end{proof}

\section{Galois theory of convergent sequences} \label{appendix:galois}

In this section we continue to work with a  field with non-trivial non-archimedean absolute value $(F,|\cdot|)$ and finite residue field. Let $\mathbb{F}$ be the completion of an algebraic closure $\overline{F}$ of $F$. The absolute value $|\cdot|$ has a unique extension to $\mathbb{F}$ that we again denote by $|\cdot|$. We consider here field extensions $F \subseteq E \subseteq E' \subseteq \overline{F}$ with $E'/E$ Galois. 

Let $A$ be an abelian variety over $F$ and let $X \subseteq A$ be a closed irreducible subvariety defined over $F$. Let $G \in \zz[T]$ be a monic polynomial of degree~$d$ with nonzero constant term, and consider the isogeny $\psi \colon A^d \to A^d$ and closed subvariety $Z \subseteq X^d$ associated to $G$ and $X$ as in \S\ref{sec:construction_inv}. For $\sigma \in \Gal(E'/E)$ and $P \in A(E')$ write \[P_\sigma = (P,\sigma(P),\ldots,\sigma^{d-1}(P)) \in A^d(E') .\]
The following lemma follows directly from the definitions.
\begin{lem} \label{lem:sigma_and_psi}
Let $\sigma\in\Gal(E'/E)$ and $P\in A(E')$ be such that $G(\sigma(P))=0$. Then, we have $\sigma(P_\sigma) = \psi(P_\sigma)$.
\end{lem}
Since in this section we will be interested in convergence properties, before proceeding we choose distance functions for closed subvarieties of $A$ and $A^d$ as in Section~\ref{sec:appendixA}.
\begin{lem} \label{lem:convergence_in_power} Let  $(P_n)_n$ be a sequence in $A(E')$ converging to $X$. Let $\sigma\in\Gal(E'/E)$ be such that for all $n$ we have $G(\sigma)(P_n) = 0$. Then, the closed subvariety $Z \subseteq X^d$ is non-empty, and the sequence $(P_{n,\sigma})_n$ in $A^d(E')$  converges to $Z$.
\end{lem}
\begin{proof}  Applying Proposition~\ref{voloch_distance}\,(vi) one sees that the sequence $(P_{n,\sigma})_n$ in $A^d(E')$ converges to $X^d$. 
Write $Y=\bigcap_{r\ge0}\psi^{*,r}(X^d)$. As $X$ is defined over $F$, so is $\psi^{*,r}(X^d)$ for each $r \ge 0$, and hence by Proposition~\ref{voloch_distance}\,(iv), Lemma~\ref{lem:sigma_and_psi}, and Proposition~\ref{voloch_distance}\,(vi), we have 
\[ \begin{split} d(P_{n,\sigma},\psi^{*,r}(X^d)) & = d(\psi^r(P_{n,\sigma}),X^d) \\ & = d(\sigma^r(P_{n,\sigma}),X^d) \\
& = d(P_{n,\sigma},X^d) . \end{split}\] This implies that for each $r \ge 0$ the sequence $(P_{n,\sigma})_n$ in $A^d(E')$  converges to $\psi^{*,r}(X^d)$. 

Let $\mathfrak{U}$ be any non-principal ultrafilter on $\nn$, and let $D$ be the ultrafield over $E'$ with respect to $\mathfrak{U}$ (see \S\ref{constr:D}). By what we have just argued, the point $P_\sigma^*\in A^d(D)$ associated to $(P_{n,\sigma})_n$ lies in $Y(D)$; in particular, the closed subvariety $Y \subseteq A^d$ is non-empty. 

As $\psi$ is proper, for all $m\ge0$ the image $\psi^m(Y)$ is closed in $A^d$. The chain \[Y\supseteq\psi(Y)\supseteq \psi^2(Y)\supseteq\dots\] must therefore stabilize by noetherianity, so that the intersection $Z$ of \eqref{eqn:Z} is a non-empty closed subset of $X^d$. Let $M\ge0$ be such that $\psi^{M}(Y)=Z$. For this $M$ we have that the
$\sigma^{M}(P_{n,\sigma})=\psi^{M}(P_{n,\sigma})$ converge to $Z$, and hence so do the $P_{n,\sigma}$ by Proposition~\ref{voloch_distance}\,(vi). Thus the sequence $(P_{n,\sigma})_n$ converges to $Z$.\end{proof}

\begin{lem} \label{lem:Galois_eqn} Consider the case that $G=(T-1)^d$.  Assume that $X$ has trivial stabilizer. Then, there exists an $M \in \zz_{>0}$ such that the following holds. Let $\sigma \in \Gal(E'/E)$  and let $(P_n)_n$ be a sequence in $A(E')$ converging to $X$. Assume that for all $n$, we have $G(\sigma)(P_n) = 0$. Let $D$ be any ultrafield over $E'$. Then, either $(P_n)_n$ converges to a proper closed subset of $X$, or the point $P^* \in X(D)$ determined by $(P_n)_n$ satisfies the equation $(\sigma^M-1)(P^*)=0$.
\end{lem}
\begin{proof} The assumption $G=(T-1)^d$ implies that $(\psi-\id)^d=0$ as an endomorphism of $A^d$, so that $\psi-\id$ is nilpotent. Since, as noted at the end of \S\ref{sec:construction_inv}, $\psi(Z)=Z$, $\psi$ permutes the finitely many irreducible components of $Z$, and so there exists an $N\in\zz_{>0}$ such that $\varphi=\psi^{N}$ fixes each irreducible component $Y$ of $Z$. It follows that $\varphi-\id$ is also nilpotent on $A^d$, as \[\varphi-\id=\psi^{N}-\id=(\psi-\id)(\psi^{N-1}+\cdots+\id)\]
is a product of $\psi-\id$ with an endomorphism commuting with it. Since $\varphi(Y)=Y$ for each irreducible component $Y$ of $Z$, we may apply Theorem~\ref{thm:dichotomy}\,(ii) to the triple $(A^d,Y,\varphi)$. This implies the existence of an $M\in\zz_{>0}$ such that $\varphi^{M}$ induces the identity map on $Y/\Stab_{A^d}(Y)^{0}$, and hence also on its finite quotient $Y/\Stab_{A^d}(Y)$, for all irreducible components $Y$ of $Z$. We will show that any such $M$ satisfies the conclusion of the lemma.

Let $\sigma \in \Gal(E'/E)$ and let $(P_n)_n$ be a sequence in $A(E')$ converging to $X$. Assume that for all $n$ we have $G(\sigma)(P_n) = 0$.
By replacing $(P_n)_n$ by a subsequence and applying Lemma~\ref{lem:convergence_in_power}, we see that there is an irreducible component $Y$ of $Z$ such that the sequence $(P_{n,\sigma})_n$ in $A^d(E')$  converges to $Y$.  Set $B = \Stab_{A^d}(Y)$.

Note that by construction of $Y$ (cf.~the beginning of the proof of Lemma \ref{lem:convergence_in_power}), we have $Y \subseteq X^d \subseteq A^d$. Let $\pr_1 \colon A^d \to A$ be the projection onto the first coordinate. By a small abuse of notation we also write $\pr_1$ for the induced map $Y \to X$. If $\pr_1(Y)$ is a proper closed subset of $X$, then $(P_n)_n=\pr_1((P_{n,\sigma})_n)$ converges to a proper closed subset of $X$.

Suppose therefore that $\pr_1 \colon Y \to X$ is surjective, so that in order to prove the lemma with our choice of $M$, it suffices to show that $(\sigma^M-1)(P^*)=0$. Then, $\pr_1(B) \subseteq \Stab_A(X)$, and the latter is trivial by assumption. As a result the map $\pr_1 \colon Y \to X$ factors through the quotient map $Y \to Y/B$. Let $\overline{\pr_1} \colon Y/B \to X$ be the resulting map. We now compute, using Lemma~\ref{lem:sigma_and_psi} and the previously shown fact that $\varphi^M$ induces the identity map on $Y/B$ (the action of which we also denote by $\varphi^M$), that
\[ \begin{split} \sigma^M(P^*) & = \pr_1(\sigma^M(P^*_\sigma)) = \pr_1(\varphi^M(P^*_\sigma)) \\
& = \overline{\pr_1}(\overline{\varphi^M(P^*_\sigma)})  = \overline{\pr_1} ( \varphi^M(  \overline{P^*_\sigma})) \\
& = \overline{\pr_1} (  \overline{P^*_\sigma})  = \pr_1(P_\sigma^*)  = P^* . 
\end{split} \]
The lemma follows.
\end{proof}

\begin{lem} \label{lem:ultra_torsion} 
Write $\Gamma = \Gal(E'/E)$ and assume that $\Gamma$ is topologically finitely generated.
Let $(P_n)_{n}$ be a sequence of torsion points of $A(E')$. 
Let $\mathfrak{U}$ be a non-principal ultrafilter and let $D'$ be the ultrafield over $E'$ with respect to $\mathfrak{U}$. Let $P^*$ be the point of $A(D')$ determined by the sequence $(P_n)_n$. 
If $P^*$ is fixed by $\Gamma$, then the set of $n \in \nn$ such that $P_n \in A(E)$ lies in $\mathfrak{U}$. 
\end{lem}
\begin{proof}By Lemma~\ref{thm:TV_origin}, the torsion subgroup $A(E')_{\tor}$ is discrete in the topological group $A(E')$.
Hence there exists $\varepsilon_0>0$ such that
\[B(O,\varepsilon_0)\cap A(E')_{\tor}=\{O\},\] where $B(0,\epsilon_0)$ is the open ball of radius $\varepsilon_0$ about $O$. For $g\in\Gamma$ and $n\in\mathbb N$, the difference $g(P_n)-P_n$ is again a torsion point, and since $d$ is
translation-invariant on $A(E')$ (see Remark~\ref{rem:translation_inv}), we have
\begin{equation}\label{eqn:uniformdist}
d(P_n,R_n)=d(O,R_n-P_n)\ge\epsilon_0
\end{equation} for any $R_n\ne P_n\in A(E')_\mathrm{tors}$. On the other hand, our hypothesis on $P^*$ implies that $d(P_n,g(P_n))\to0$ along $\mathfrak{U}$, i.e., there is a subsequence $P_{i_n}$ whose indexing set belongs to $\mathfrak{U}$ along with a sequence of positive reals $\varepsilon_{i_n}\to0$ such that $d(P_{i_n},g(P_{i_n}))=\varepsilon_{i_n}$ and $\varepsilon_{i_n}<\varepsilon_0$. But taking $R_n=g(P_n)$ in \eqref{eqn:uniformdist}, this forces $P_n=g(P_n)$ for $\mathfrak{U}$-almost all $n$.

In particular, letting $g_1,\ldots,g_k$ be topological generators of $\Gamma$ and recalling that $\Sigma_1,\dots,\Sigma_\ell\in\mathfrak{U}\Longrightarrow\Sigma_1\cap\cdots\cap\Sigma_\ell\in\mathfrak{U}$, we see that $P_n = g_1(P_n) = \cdots = g_k(P_n)$ for $\mathfrak{U}$-almost all $n$. Hence $P_n\in A(E)$ for $\mathfrak{U}$-almost all $n$.\end{proof}

\section{A Tate--Voloch type result in characteristic~$p$} \label{sec:appendixB}

Let $k$ be a finite field, and let $K$ be a field of transcendence degree one over~$k$. We assume that $k$ is algebraically closed in $K$ and fix a place $v$ of $K/k$. Let $C/k$ be a smooth projective normal curve with function field $K$. We view $v$ as a closed point of $C$.

We work here with abelian varieties $A$ defined over $K$, and closed irreducible subvarieties $X \subseteq A$ similarly defined over $K$.  Our aim is to show that any such $X$ is unapproximable (see Definition~\ref{def:bad}) by torsion points of order prime to $p$ in the $v$-adic topology.

Denote by $K_v$ the completion of $K$ at $v$ and let $\cc_v$ be the completion of an algebraic closure of $K_v$.
The following lemma is clear from Proposition~\ref{voloch_distance}\,(iv).

\begin{lem} \label{lem:reduction_trivial_stab} Let $X \subseteq A$ be any closed subvariety. Let $\Sigma \subseteq A(\cc_v)$ be a subset. Let $B=\mathrm{Stab}_A(X)$ and denote by $p \colon A \to A/B$ the projection map. If $X/B$ is unapproximable by $p(\Sigma)$, then $X$ is unapproximable by $\Sigma$.
\end{lem}

\subsection{The case of special subvarieties}

Let $\overline{K}$ be an algebraic closure of $K$, and let $\overline{k}$ be the algebraic closure of $k$ inside $\overline{K}$. We first show that special subvarieties of $A$ are unapproximable by torsion points. 
\begin{thm} \label{thm:TV_special_case} 
Assume that $X$ is a $\overline{K}/\overline{k}$-special subvariety of $A$. 
Then, $X$ is unapproximable by torsion points of $A(\mathbb{C}_v)$.
\end{thm}
\begin{proof}  

Let $(A^{\overline{K}/\overline{k}},\Tr^{\overline{K}/\overline{k}}_A)$ be the $\overline{K}/\overline{k}$-trace of $A$.
Since $X$ is $\overline{K}/\overline{k}$-special, there exist an abelian subvariety $G\subseteq A$,
a torsion point $\tau\in A(\overline{K})$, and a closed subvariety $X'\subseteq A^{\overline{K}/\overline{k}}$ such that
\[X=G+\Tr^{\overline{K}/\overline{k}}_A(X'\otimes_{\overline{k}}\overline{K})+\tau,\]
and one has \[\Stab_A(X)=G+\Tr^{\overline{K}/\overline{k}}_A\!\bigl(\Stab_{A^{\overline{K}/\overline{k}}}(X')\bigr).\]
By Lemma~\ref{lem:stab_quotient} and Lemma~\ref{lem:reduction_trivial_stab}, we may replace $A$ by
$A/\Stab_A(X)$ and $X$ by $X/\Stab_A(X)$, thereby assuming that $G=0$ and thus that \[X=X_0+\tau\quad\textup{ where }X_0=\Tr^{\overline{K}/\overline{k}}_A(X'\otimes_{\overline{k}}\overline{K}).\] Let  $d_v(\cdot,X_0) \colon A(\cc_v) \to \rr_{\ge 0}$ be a $v$-adic distance function with respect to~$X_0$.

\noindent{\textbf{Claim~1.}} Let $P \in  \Tr^{\overline{K}/\overline{k}}_A (A^{\overline{K}/\overline{k}}(\overline{k})) $ and assume that $P \notin \Tr^{\overline{K}/\overline{k}}_A(X'(\overline{k}))$. Then, we have $d_v(P,X_0) =1$. 

\noindent{\emph{Proof of Claim~1.}} We have that $X' \times_k C$ is a model of $X' \otimes_k K$ over $C$, and $A^{\overline{K}/\overline{k}} \times_k C$ is a model of $A^{\overline{K}/\overline{k}} \otimes_k K$ over $C$. 

Let $x \in A^{\overline{K}/\overline{k}}(\overline{k})$ be any point and set $P= \Tr^{\overline{K}/\overline{k}}_A(x)$. Assume that $P \notin \Tr^{\overline{K}/\overline{k}}_A(X'(\overline{k}))$, so that $x \notin X'(\overline{k})$. Then, the closed subvariety $\{x\} \times C$ of $A^{\overline{K}/\overline{k}} \times_k C$ is disjoint from the closed subvariety $X' \times_k C$.  This implies the claim. \qed

Let $\mathfrak{t} \subseteq A $ be the image of $\Tr^{\overline{K}/\overline{k}}_A$. It is an abelian subvariety of $A$ containing $X_0$. 

\noindent{\textbf{Claim~2.}} When $P \in \mathfrak{t}(\overline{K})$ is torsion, we have $P \in \Tr^{\overline{K}/\overline{k}}_A (A^{\overline{K}/\overline{k}}(\overline{k}))$. 

\noindent{\emph{Proof of Claim~2.}}  Following \cite[Theorem~6.12]{conrad_chow}, the map $\Tr^{\overline{K}/\overline{k}}_A \colon A^{\overline{K}/\overline{k}} \otimes_k \overline{K} \to A$ is a purely inseparable isogeny onto $\mathfrak{t}$. It therefore induces an isomorphism $(A^{\overline{K}/\overline{k}} \otimes_k \overline{K})(\overline{K}) \isom \mathfrak{t}(\overline{K})$. Next, the natural inclusion $A^{\overline{K}/\overline{k}}(\overline{k}) \subseteq (A^{\overline{K}/\overline{k}} \otimes_k \overline{K})(\overline{K})$ is an isomorphism on torsion points. Combining these facts, we see that the map $A^{\overline{K}/\overline{k}}(\overline{k})_\tor \to \mathfrak{t}(\overline{K})_\tor$ induced from $\Tr^{\overline{K}/\overline{k}}_A$  is an isomorphism. \qed

From these Claims, it follows that for all torsion points $P \in (\mathfrak{t} \setminus X_0)(\overline{K})$, we have  $d_v(P,X_0) = 1$.
Now suppose for a contradiction that there exists a sequence of torsion points of $(A \setminus X)(\overline{K})$ whose distances to $X$ converge to $0$. Via translation by the torsion point $-\tau$, we obtain a sequence $(P_n)_n$ of torsion points of $(A \setminus X_0)(\overline{K})$ whose distances to $X_0$ converge to $0$. By our observation above, by passing to a subsequence we obtain
a sequence of torsion points of $(A \setminus \mathfrak{t})(\overline{K})$ whose distances to $\mathfrak{t}$ converge to~$0$. Passing to the quotient abelian variety $A/\mathfrak{t}$, we obtain a sequence of nonzero torsion points of $(A/\mathfrak{t})(\overline{K})$ whose distances to the origin converge to $0$. This contradicts the discreteness of torsion as asserted in Lemma~\ref{thm:TV_origin}.\end{proof}

\subsection{The case of unramified prime-to-$p$ torsion}

Recall that our aim is to show that any closed subvariety $X$ of any abelian variety $A$ over $K$ is unapproximable by torsion points of order prime to $p$ in the $v$-adic topology. Here $v$ is a given place of $K$. 

Given this aim, there is no loss of generality if we replace $K$ by a finite field extension. Thus by the semi-abelian reduction theorem we may assume that $K$ has semi-abelian reduction at the place $v$. We may also assume that our given finite field $k$ is the residue field of $K$ at $v$. These two assumptions will be in place from now on.  

Let $K_v^\unr$ denote the maximal unramified extension of $K_v$. 
We have a canonical isomorphism of topological groups $\Gal(K_v^\unr / K_v) \cong \Gal(\overline{k}/k)$. Let $\tau \in\Gal(K_v^\unr /K_v)$ correspond to the Frobenius automorphism in $\Gal(\overline{k}/k)$ through this isomorphism.
\begin{lem} \label{lem:weil} 
There exists a monic polynomial $G \in \zz[T]$ having no roots of unity among its complex roots and nonzero constant term such that all torsion points $P\in A(K_v^\unr)$ of order prime to $p$ satisfy $G(\tau)(P) = 0$.
\end{lem}
\begin{proof} Let $\mathcal{N}$ be the N\'eron model of $A$ over~$C$ and let $\mathcal{N}_v$ be the fiber of $\mathcal{N}$ at~$v$.  By our assumption that $A$ has semi-abelian reduction at $v$, the fiber $\mathcal{N}_v$ is a finite union of semi-abelian varieties over the finite field~$k$. By the Riemann hypothesis for semi-abelian varieties, there exists a monic polynomial $G \in \zz[T]$ having no roots of unity among its complex roots and with nonzero constant term such that $G(\Frob)$ vanishes in $\End(\mathcal{N}_v \times_k \overline{k})$. Let $\pi \colon A(K_v^\unr) \to \mathcal{N}_v(\overline{k})$ be the reduction map, and let $P \in A(K_v^\unr)$ be a torsion point of order prime to $p$. As $G(\Frob)$ vanishes in $\End(\mathcal{N}_v \times_k \overline{k})$, we have $\pi(G(\tau)(P))=G(\Frob)(\pi(P))=0$. By Lemma~\ref{lem:kernel_reduction} and the fact that $P$ has order prime to $p$, this forces $G(\tau)(P)=0$.\end{proof}

\begin{thm} \label{thm:tate-voloch-unr} 
Each closed subvariety $X\subseteq A$ is unapproximable by torsion points of prime-to-$p$ order contained in $A(K_v^\unr)$.
\end{thm}
\begin{proof} 
As argued in the proof of Theorem \ref{thm:TV_special_case}, Lemma~\ref{lem:reduction_trivial_stab} implies that by passing to the quotient by the stabilizer of $X$, we may assume that $X$ has trivial stabilizer. Let $(P_n)_n$ be a sequence of torsion points in $A(K_v^{\unr})$ of prime-to-$p$ order converging to $X$. By Lemma~\ref{lem:weil}, there is  a monic polynomial $G \in \zz[T]$ having no roots of unity amongst its complex roots and with nonzero constant term such that $G(\tau(P_n))=0$ for all $n\in\mathbb{N}$. 

Let $d$ be the degree of~$G$. For $P \in A(K_v^\unr)$ we write $P_\tau = (P,\tau(P),\ldots,\tau^{d-1}(P))$.
Let $\psi \colon A^d \to A^d$ be the isogeny associated to $G$, and let $Z \subseteq X^d$ be the associated closed subvariety of $A^d$ as in \eqref{eqn:Z}. By Lemma~\ref{lem:convergence_in_power}, we have that $Z$ is non-empty and that the sequence $(P_{n,\tau})_n$ converges to $Z$. Moreover, as noted in \S\ref{appendix:invariant}, we have $\psi(Z)=Z$. As $Z$ has only finitely many irreducible components, there is an irreducible component $Y$ of $Z$ and a subsequence $(P_{n_k,\tau})_k$ of $(P_{n,\tau})_n$ that converges to $Y$. Note that $Y$ can be defined over $\overline{K}$.

Let $\varphi= \psi^m$ be a power of $\psi$ with $\varphi(Y)=Y$.
The characteristic polynomial of $\psi $ is a power of $G$; in particular, none of its complex roots is a root of unity. The roots of the characteristic polynomial of $\varphi$ are powers of the roots of the characteristic polynomial of $\psi$, hence none of its complex roots is a root of unity. On the other hand, the characteristic polynomial of the induced isogeny from $\varphi$ on any positive dimensional $\varphi$-invariant abelian subvariety of $A$ is a factor of the characteristic polynomial of $\varphi$. Hence none of the  complex roots of the characteristic polynomial of the induced isogeny from $\varphi$ on a positive dimensional $\varphi$-invariant abelian subvariety of $A$ is a root of unity. The characteristic polynomial of a finite order automorphism of an abelian variety has only roots of unity as roots. We conclude that $A^d$ has no positive-dimensional abelian subvarieties that are pure of weight $0$ with respect to the isogeny $\varphi$. By Theorem~\ref{thm:dichotomy}\,(i), it follows that $Y$ is a $\overline{K}/\overline{k}$-special subvariety of $A^d$. Applying Theorem~\ref{thm:TV_special_case} yields $P_{n,\tau}\in Y$ and hence $P_n\in X$ for all sufficiently large $n$.\end{proof}

\subsection{The case of arbitrary prime-to-$p$ torsion}

We finally consider arbitrary prime-to-$p$ torsion. 
We start by noting that all prime-to-$p$ torsion of $A(\overline{K})$ is contained in $A(K^\sep)$ where $K^\sep \subseteq \overline{K}$ is the separable closure of $K$ in $\overline{K}$. Indeed, for any $n$ with $\gcd(n,p)=1$, the group scheme $A[n]$ is finite \'{e}tale over $K$. 

\begin{lem} \label{lem:log_monodromy} 
Let $P \in A(K^\sep)$ be a torsion point of order prime to $p$, and let $\sigma \in \Gal(K_v^\sep/K_v^\unr)$. Then, $(\sigma-1)^2(P)=0$.
\end{lem}
\begin{proof}  Recall that we are assuming that $A$ has semi-abelian reduction at $v$. The result follows from \cite[Proposition~IX.3.5]{SGA_7.I}.\end{proof}

Let $K_v \subseteq K_v^\tame$ denote the maximal tamely ramified extension of $K_v$ in $K_v^\sep$. 

\begin{lem} \label{lem:prime_to_p_tame_ram} The prime-to-$p$ torsion of $A(\overline{K})$ is contained in $A(K_v^\tame)$.
\end{lem}
\begin{proof} The proof is elementary \cite[Lemma~A.1]{baker-ribet}. We include it here for the purposes of self-containment.
\par
By Lemma~\ref{lem:log_monodromy}, for every prime-to-$p$ torsion point
$P\in A(K^{\sep})$ and every $\sigma\in\Gal(K_v^{\sep}/K_v^{\tame})$ we have
$(\sigma-1)^2(P)=0$.  Set $N=\sigma-1$, so that $N^2(P)=0$ and
$\sigma(P)=(1+N)(P)$ for every prime-to-$p$ torsion point
$P\in A(K^{\sep})$ .

Let $P\in A(K^{\sep})$ be a prime-to-$p$ torsion point of order $m$, so that $(m,p)=1$. Since $mP=0$, we also have
$mN(P)=N(mP)=0$.  Using binomial expansion and $N^2(P)=0$ we obtain \[\sigma^{m}(P)=(1+N)^{m}(P)=(1+mN)(P)=P+m N(P)=P.\]

On the other hand, $\Gal(K_v^{\sep}/K_v^{\tame})$ is a pro-$p$ group.
Hence the action of $\sigma$ on the finite group $\langle P\rangle$,
whose order equals $m$ and is in particular prime to $p$, must be trivial.
Therefore $\sigma(P)=P$ for all $\sigma\in\Gal(K_v^{\sep}/K_v^{\tame})$,
and hence $P\in A(K_v^\tame)$.\end{proof}

\begin{lem} \label{lem:for_induction} Assume $X$ has trivial stabilizer.  
Let $(P_n)_{n}$ be a sequence in $A(\overline{K})$ that converges to $X$. Let $N \in \zz_{>0}$.  Then, the points $N\cdot P_n$ converge to the closed subvariety $N\cdot X$. Moreover, either a subsequence of $(P_n)_{n}$ converges to a proper closed subset of $X$, or the set of $n \in \nn$ such that $N.P_n \notin N\cdot X$ is infinite.
\end{lem}
\begin{proof} The first statement follows from Proposition~\ref{voloch_distance}\,(iv). Let \[V = \{ n \in \nn : N.P_n  \notin N\cdot X\}.\] If $V$ is finite, then $N.P_n  \in N\cdot X$ for all but finitely many $n \in \nn$. For those $n$ we have $P_n \in X + \Ker[N]$, and since $\Ker[N]$ is finite, there exists an $h \in \Ker[N]$ and a subsequence of $(P_n)_{n \in \nn}$ that is contained in $X+h$. This subsequence converges to $(X+h) \cap X$. Since by assumption $X$ has trivial  stabilizer, the latter is a proper closed subset of $X$.
\end{proof}
The following is our Tate--Voloch type result in positive characteristic.
\begin{thm} \label{thm:tate-voloch-app}  
Each  closed subvariety $X\subseteq A$ is unapproximable by torsion points of $A(\mathbb{C}_v)$ of order prime to $p$.
\end{thm}
\begin{proof} We may assume that $X$ is irreducible. By Lemma~\ref{lem:reduction_trivial_stab}, upon passing to the quotient by the stabilizer of $X$ we may assume $X$ has trivial stabilizer. We proceed by induction on $\dim(X)$. When $\dim(X)=0$, the statement follows from the discreteness of torsion as asserted in Lemma~\ref{thm:TV_origin}. Assume now that the statement is proven for $X$ with $\dim(X) < d$, and consider an $X$ with $\dim(X)=d$. 

Let $(P_n)_{n}$ be a sequence of torsion points of $(A \setminus X)(\overline{K})$ of order prime to $p$ converging to $X$. By Lemma~\ref{lem:prime_to_p_tame_ram}, for all $n \in \nn$ we have $P_n \in A(K_v^\tame)$. The extension $K_v^\unr \subseteq K_v^\tame$ is Galois. By Lemma~\ref{lem:log_monodromy}, for all $\sigma \in \Gal(K_v^\tame / K_v^\unr)$ and all $n \in \nn$, we have $(\sigma-1)^2(P_n)=0$. 

By Lemma~\ref{lem:Galois_eqn} applied with $G=(T-1)^2$, there exists an $M \in \zz_{>0}$ such that for all $\sigma \in \Gal(K_v^\tame/K_v^\unr)$ and all ultrafields $D$ over $K_v^\tr$ with respect to a non-principal ultrafilter, either the point $P^* \in X(D)$ associated to $(P_n)_{n}$ satisfies $(\sigma^M-1)(P^*)=0$, or $(P_n)_n$ converges to a proper closed subset of $X$. By the induction hypothesis, we may without loss of generality assume the former.

Let $N \in \zz_{>0}$ be such that $N(T-1)$ is in the ideal of $\zz[T]$ generated by the polynomials $(T-1)^2$ and $(T^M-1)$. As we have assumed without loss of generality that $(P_n)_n$ does not converge to a proper closed subset of $X$, Lemma~\ref{lem:for_induction} implies that the $N.P_n$ converge to $N.X$ and the set of $n$ such that $N.P_n\notin N. X$ is infinite. 

Let $\mathfrak{U}$ be a non-principal ultrafilter such that the set of $n$ with $N.P_n \notin N. X$ lies in $\mathfrak{U}$, and consider the point $P^*$ determined by the associated ultrafield $D$. Then, by our choice of $N$ and the above argument that for all $\sigma \in \Gal(K_v^\tame/K_v^\unr)$, \[(\sigma^M-1)(P^*)=0=(\sigma-1)^2(P^*),\] it follows that $(\sigma-1)(N.P^*)=0$ for all such $\sigma$. 

The Galois group $\Gal(K_v^\tame/K_v^\unr)$ is isomorphic to a quotient of $\widehat{\zz}$; in particular it is topologically finitely generated. By Lemma~\ref{lem:ultra_torsion} and the fact that $(\sigma-1)(N.P^*)=0$ for all $\sigma \in \Gal(K_v^\tame/K_v^\unr)$, the set of $n \in \nn$ such that $N.P_n \in A(K_v^\unr)$ lies inside $\mathfrak{U}$. We thereby obtain a sequence of torsion points of $(A \setminus N.X)(K_v^\unr)$ of order prime to $p$ converging to $N.X$. This contradicts Theorem~\ref{thm:tate-voloch-unr}. 
   \end{proof}

\part{An analogue of Ih's conjecture over function fields} \label{part_two}

The aim of this part is to give our proof of the ~\mainthmrefcolor. We need some preparatory material concerning local and global canonical heights on abelian varieties over function fields.
In \S\ref{sec:normalized} and \S\ref{sec:can_heights_points}, we recall the construction and main properties of these heights, including a review of relevant material from the recent work \cite{djs_canonical} on normalized local canonical heights. In \S\ref{sec:can_ht_subvar}, we discuss canonical heights of divisors and present a Mahler measure-type formula for them based on a recursion formula for arithmetic intersections, going back to work of Chambert-Loir and Thuillier \cite{clt}. We also review the geometric Bogomolov conjecture for divisors shown by Yamaki \cite{yamaki}.
In \S\ref{sec:log_equid}, we state and prove our ``logarithmic equidistribution'' result for Galois orbits of torsion points of order prime to $p$, using an equidistribution result due to Faber \cite{faber_equid} and Gubler \cite{gub_equid} and Theorem~\ref{thm:tate-voloch-app} as key inputs. 
In \S\ref{sec:proof_last_reduction_step}, we prove the ~\mainthmrefcolor. 

\subsection*{Notation and terminology}
Throughout Part \ref{part_two} of this paper $K$ will denote a field of transcendence degree one over an algebraically closed field $\mathds{k}$. Starting halfway in \S\ref{sec:log_equid}, the field $\mathds{k}$ will be an algebraic closure of a finite field. We will denote by $\overline{K}$ an algebraic closure of $K$. The set of places of $K/\mathds{k}$ will be denoted by $M_K$. 
We denote by $C$ the unique normal projective integral curve over $\mathds{k}$ with function field $K$. We identify $M_K$ with the set of closed points of $C$. 

If $X$ is a variety over $K$, a \emph{model} of $X$ is any finite type scheme $\mathcal{X}/C$ with generic fiber identified with $X$. If $\mathcal{X}, \mathcal{X}'$ are two models of $X$, there exist a non-empty open subset $U$ of $C$ and a $U$-isomorphism $\mathcal{X} \times_C U \cong \mathcal{X}' \times_C U$.

Fix an effective divisor $D$ on $X$ defined over $K$, a point $x \in X(K)$, a model $\mathcal{X}/C$ of $X$, and a place $v \in M_K$. We say that $x$ is \emph{integral} with respect to $D$ and $\mathcal{X}$ at $v$ if the Zariski closures of $x$ and $|D|$ over $\mathcal{X}$ are disjoint over~$v$. When $S$ is a set of places of $K$, we say that $x$ is \emph{$S$-integral} with respect to $D$ and $\mathcal{X}$ if $x$ is integral with respect to $D$ and $\mathcal{X}$ at all places outside $S$. These notions of integrality with respect to a given model $\mathcal{X}/C$ readily extend to effective divisors $D$ on $X$ defined over $\overline{K}$ and points $x \in X(\overline{K})$.

Given a place $v \in M_K$, we denote by $K_v$ the completion of $K$ at $v$. We fix an algebraic closure $\overline{K_v}$ of $K_v$ and denote its completion by $\cc_v$. For every $v\in M_K$, we choose the absolute value $|\cdot|_v$  on $K_v$ given by $|a|_v = \mathrm{e}^{-\ord_v(a)}$ for $a \in K^\times$, and extend the absolute value $|\cdot|_v$ to $\cc_v$. The collection of absolute values $\{|\cdot|_v\}_{v \in M_K}$ satisfies the product formula.

 When $X$ is an irreducible variety over~$K$ and $v \in M_K$, we will denote by $X_v^\an$ the Berkovich analytification of $X$ over $\cc_v$. The analytic space $X_v^\an$ naturally contains $X(\cc_v)$ with its non-archimedean analytic topology as a subspace.  

When $A$ is an abelian variety over $K$, the space $A_v^\an$ also contains a \emph{canonical skeleton}, denoted by $\varSigma_v$. The space $\varSigma_v$ has a natural structure of real torus and the Berkovich analytic space $A_v^\an$ deformation retracts onto $\varSigma_v$. We denote by $\mu_{H,v}$ the Haar measure of $\varSigma_v$ normalized to give $\varSigma_v$ unit volume. We refer to \cite{djs_canonical} and the references therein for further details about these notions.

We denote by $\log$ the natural logarithm. 

\section{Normalized canonical local heights} \label{sec:normalized}

The purpose of this section is to review the notion of canonical local heights on abelian varieties over function fields. We also discuss the notion of \emph{normalized} canonical local heights, following~\cite[\S7]{djs_canonical}.

Let $A$ be an abelian variety over $K$. Let $L$ be a line bundle on $A$. We will assume, for the purposes of exposition, that $L$ is \emph{effective}, though this is not essential. 

We fix a place $v \in M_K$. We have the analytification $L_v^\an$ of $L$ over the analytification $A_v^\an$ of $A$ at $v$. Let $\|\cdot\|$ be any continuous metric on $L_v^\an$ and let $s \in H^0(A,L)$ be any nonzero global section of $L$. Set $D = \divisor_L(s)$. We call the function 
\begin{equation} 
\lambda_v(x) = -\log \|s(x)\| \, , \quad x \in A_v^\an  \setminus |D_v^\an| 
\end{equation}
a \emph{local height} on $A_v^\an$ with respect to the effective divisor~$D$.

The function $\lambda_v$ is continuous and bounded below on $A_v^\an \setminus |D_v^\an|$ and diverges to $+\infty$ near $|D_v^\an|$. In fact, it defines a \emph{Green's function} with respect to~$D$ in the sense of \cite[\S2]{clt}.  In particular, let $U \subset A$ be a Zariski open set where $D$ is given by a local equation $z$. Then, we have
 \begin{equation} \label{eqn:local_eqn}
 \lambda_{v}(P) = -\log|z(P)|_v + O(1) \, , \quad P \in (U \setminus |D|)(\cc_v) \, , 
 \end{equation}
 where $O(1)$ denotes a bounded continuous function on $U(\cc_v)$.

We will assume, from now on, that the line bundle $L$ is \emph{rigidified}, i.e. it comes equipped with a trivialization along the origin. Assume that the rigidified line bundle $L$ is \emph{symmetric} (resp.\ \emph{anti-symmetric}). Then, for all nonzero $n \in \zz$, we have a unique isomorphism of rigidified line bundles $[n]^*L \isom L^{\otimes n^2}$ (resp.\ $[n]^*L \isom L^{\otimes n}$). The \emph{canonical metric} on $L_v^\an$ is the unique continuous metric on $L_v^\an$ such that, for all nonzero $n \in \zz$, the isomorphism $[n]^*L \isom L^{\otimes n^2}$ (resp.\ $[n]^*L \isom L^{\otimes n}$) is an isometry. 
The canonical metric exists by \cite[Example~3.7]{gu}. For arbitrary rigidified line bundles $L$, one obtains a canonical metric from the canonical metrics on the symmetric rigidified line bundle $L \otimes [-1]^*L$ and the anti-symmetric rigidified line bundle $L \otimes [-1]^*L^{\otimes -1}$. We denote the canonical metric on $L_v^\an$ by $\|\cdot\|_{L,v}$.

The local height function 
\begin{equation} \label{eqn:initial}
\lambda_{D,v}(x) = -\log \|s(x)\|_{L,v} \, , \quad x \in A_v^\an  \setminus |D_v^\an| 
\end{equation}
associated to the canonical metric is called the \emph{canonical local height} on $A_v^\an$ with respect to the rigidified line bundle~$L$ and the section~$s$. Although this is not visible from the notation, the function $\lambda_{D,v}$ depends (by an additive constant) on the choice of the global section~$s$ defining the divisor $D$. 

By pushforward along the inclusion the normalized Haar measure $\mu_{H,v}$ on $\varSigma_v$ determines a measure on $A_v^\an$, which we also denote by $\mu_{H,v}$. As the restriction  to $\varSigma_v$ of the local height  $\lambda_{D,v}$ from \eqref{eqn:initial} is continuous, we see that the canonical local height $\lambda_{D,v}$ is $\mu_{H,v}$-integrable. 

In particular, we may renormalize the function $\lambda_{D,v}$ from \eqref{eqn:initial} by defining
\begin{equation} \label{eqn:def_our_normalized}
\lambda'_{D,v}(x) = -\log \|s(x)\|_{L,v} +  \int_{A_v^\an} \log \|s\|_{L,v} \, \d \, \mu_{H,v} \, , \quad x \in A_v^\an \setminus |D_v^\an| .
\end{equation}
The resulting function $\lambda'_{D,v}$ is now intrinsic to the divisor~$D$. In other words, changing the global section $s$ that defines $D$ by a scalar multiple or changing the rigidification of $L$ do not change the function $\lambda'_{D,v}$. In \cite{djs_canonical},  the terminology \emph{normalized canonical local height} is proposed for the function $\lambda'_{D,v}$ from \eqref{eqn:def_our_normalized}.

\begin{example} Assume that $A$ is an elliptic curve and let $O$ denote its origin. It is shown in \cite[Proposition~8.1]{djs_canonical} that the restriction of the normalized canonical local height $\lambda'_{O,v}$ to $A(\overline{K_v})$ coincides with Tate's normalization of the canonical local height on $A(\overline{K_v})$, as constructed in \cite[Chapter~VI]{sil_advanced}. 
\end{example}

\subsection{The case of good reduction}

In \cite{djs_canonical}, an explicit formula is given for the normalized canonical local height $\lambda'_{D,v}$. Here, we review the formula, when $v$ is a place of \emph{good reduction}. Let $\N$ be the N\'eron model of $A$ over $C$. For every $x \in A(K) \setminus |D|$, we let $\mathbf{i}_v(x,D)$ denote the intersection multiplicity between the closure of $x$ and the closure of $D$ on $\N$ above the closed point $v$ of $C$.
\begin{thm} \label{thm:normalized_good} Assume that $A$ has good reduction at $v$ and let $x \in A(K) \setminus |D|$. Then, the formula
\[ \lambda'_{D,v}(x) = \mathbf{i}_v(x,D) \]
holds. 
\end{thm}
\begin{proof} See \cite[Corollary~7.3]{djs_canonical}.
\end{proof}
\begin{cor} \label{cor:normalized_good} Assume that $A$ has good reduction at $v$. Let $x \in A(K) \setminus |D|$ and assume that $x$ is 
integral with respect to $D$ and the N\'eron model~$\N$ at $v$. Then, we have $\lambda'_{D,v}(x) = 0$.
\end{cor}

\section{N\'eron--Tate heights}\label{sec:can_heights_points}

We continue with the notations and assumptions of the previous section.
In particular, we let $L$ be a line bundle on $A$. As is well known (see, for example, \cite[Corollary~9.2.5]{bg_heights}), in the class of Weil heights $\h_L \colon A(\overline{K}) \to \rr$ associated to $L$ there exists a unique function $\h'_L \colon A(\overline{K}) \to \rr$ that can be written as the sum of a quadratic form and a linear form. We call $\h'_L$ the \emph{N\'eron--Tate height} associated to~$L$.

\begin{prop} \label{prop:standard_height}
Let $L$ be a line bundle on $A$. We have the following properties of the N\'eron--Tate height.
\begin{itemize}
\item[(i)] For all $\sigma \in \Aut(\overline{K} / K)$ and all $x \in A(\overline{K})$ we have $\h'_L(\sigma(x)) = \h'_L(x)$.
\item[(ii)] Let $x \in A(\overline{K})$ be a torsion point. We have $\h'_L(x)=0$.
\end{itemize}
\end{prop}
\begin{proof} Property (i) is clear from the definition. For (ii), write $\h'_L= q + l$, where $q \colon A(\overline{K}) \to \rr$ is a quadratic form and $l \colon A(\overline{K}) \to \rr$ is a linear form. Let $m \in \zz_{>0}$ and $x \in A(\overline{K})$ be such that $[m](x) = 0$. Then, we obtain
\[ 0 = m^2 q(x) + m l(x) \, , \quad 0 = 4m^2 q(x) + 2m l(x).\]
We conclude $q(x) = l(x) = 0$ and, hence, $\h'_L(x)=0$.\end{proof}

Equip $L$ with a rigidification at the origin. 
As it turns out, the N\'eron--Tate height $\h'_L \colon A(\overline{K}) \to \rr$ is equal to the global height associated to the canonical metrics $\|\cdot\|_{L,v}$ on $L_v^\an$ at all places  $v \in M_K$, see for instance \cite[Corollary~9.5.14]{bg_heights}. 

More explicitly, assume that $L$ is effective, and let $s$ be a nonzero global section of $L$, with divisor~$D$. 
Then, for each $x \in A(K) \setminus |D|$, we have the local-to-global formula 
\begin{equation} \label{eq:local_global}
\h'_L(x) =    \sum_{v \in M_K} \lambda_{D,v}(x)  \, ,
\end{equation}
with $\lambda_{D,v}$ as in \eqref{eqn:initial}.
This formula can be rewritten as
\begin{equation} \label{eq:local_global_bis}
\h'_L(x) =   \sum_{v \in M_K}    \lambda'_{D,v}(x) -  \sum_{v \in M_K}  \int_{A_v^\an} \log \|s\|_{L,v} \, \d \, \mu_{H,v} 
\end{equation}
with $\lambda'_{D,v}$ as in \eqref{eqn:def_our_normalized}.
Theorem~\ref{thm:normalized_good} ensures that both sums in \eqref{eq:local_global_bis} are finite. 

We will need a generalization of \eqref{eq:local_global_bis} to the setting where $x$ is only defined over $\overline{K}$. 
For each $v \in M_K$, we fix a $K$-embedding $\iota_v : \overline{K} \hookrightarrow \overline{K_v}$. We denote by  
 $\epsilon_v : A(\overline{K}) \hookrightarrow A(\overline{K_v})$ the induced map. For $x \in A(\overline{K})$, we denote by $\Orb(x)$ the $\Aut(\overline{K}/K)$-orbit of $x$.

\begin{lem} \label{lem:Galois_orbits}
Let $K \subseteq E \subseteq \overline{K}$ be extensions of fields, and assume $E/K$ is  finite and normal. Let $x \in A(E) \setminus |D|$. Let $v \in M_K$, and let $w \in M_E$ be any place of $E$ lying above $v$. Then, for every $y \in \Orb(x)$, we have $y \in A(E) \setminus |D|$. Moreover,
\[\{ \lambda'_{D,w}(y):y \in\Orb(x) \}=\{ \lambda'_{D,v}(\epsilon_v(y)):y\in \Orb(x) \}.\]
\end{lem}
\begin{proof}
The first assertion follows from the assumption that $E/K$ is normal and the assumption that $D$ is defined over~$K$. Let $w_0 \in M_E$ be the place corresponding to the restriction of the chosen $K$-embedding $\iota_v : \overline{K} \hookrightarrow \overline{K_v}$ to $E$. Then, for every $y \in A(E) \setminus |D|$, we have
\[\lambda'_{D,w_0}(y) = \lambda'_{D,v}(\epsilon_v(y)).\] 
In particular,
\[\{\lambda'_{D,w_0}(y):y\in \Orb(x)\}=\{\lambda'_{D,v}(\epsilon_v(y)) : y \in \Orb(x)\}.\] 
For any $\sigma \in \Aut(\overline{K}/K)$ and any $w \in M_E$ with $w \mid v$ we have
\[\begin{aligned}
\{ \lambda'_{D,w}(y) : y \in \Orb(x)\}
&=\{ \lambda'_{D,w}(\tau(x)) : \tau \in \Aut(\overline{K}/K)\} \\
&=\{ \lambda'_{D,\sigma(w)}(\sigma\tau(x)) : \tau \in \Aut(\overline{K}/K)\} \\
&=\{ \lambda'_{D,\sigma(w)}(y):y\in \Orb(x)\}.
\end{aligned}\]
The proof is finished by noting that $\Aut(\overline{K}/K)$ acts transitively on the set of places of $E$ that extend~$v$. This is well known in the case that $E/K$ is Galois (see, for example, \cite[Corollary XII.4.10]{lang_algebra}). The general case follows from this by considering the separable closure of $K$ inside $E$ and noting that absolute values have unique extensions through purely inseparable field extensions. 
\end{proof}

\begin{prop} \label{prop:local_to_global}
Let $x \in A(\overline{K}) \setminus |D|$.  The formula
\[
\h'_L(x) \;=\;   \sum_{v \in M_K}   \frac{1}{\# \Orb(x)} \sum_{y \in \Orb(x)} \lambda'_{D,v}(\epsilon_v(y)) \;-\;  \sum_{v \in M_K}  \int_{A_v^\an} \log \|s\|_{L,v} \, \mathrm{d}\mu_{H,v}
\]
holds. Both sums over $M_K$ are finite sums.
\end{prop}

\begin{proof} The fact that both sums over $M_K$ are finite follows from Theorem~\ref{thm:normalized_good}.
Let $K \subseteq E \subseteq \overline{K}$ be field extensions with $E/K$ finite normal and suppose that $x \in A(E) \setminus |D|$. By Proposition~\ref{prop:standard_height}\,(i), for each $y \in \Orb(x)$, we have $\h'_L(y)= \h'_L(x)$. From \eqref{eq:local_global}, for each $y \in \Orb(x)$, we have
\begin{equation} \label{eqn:height_extension}
\h'_L(y) =  \sum_{w \in M_E} \frac{[ E_w : K_v] }{[E:K]} \lambda_{D,w}(y) \, ,
\end{equation}
where for each $w \in M_E$ we denote by $v$ the prime of $K$ under $w$. We denote
\[
c =  \sum_{v \in M_K}  \int_{A_v^\an} \log \|s\|_{L,v} \, \mathrm{d}\mu_{H,v} .
\]
We compute:
\[
\begin{split}
[E : K] \, \h'_L(x)
&= \frac{1}{\# \Orb(x)} \sum_{y \in \Orb(x)} [E:K] \, \h'_L(y) \\
&=  \frac{1}{\# \Orb(x)} \sum_{y \in \Orb(x)}  \sum_{w \in M_E}  [E_w:K_v] \,  \lambda'_{D,w}(y) - [E:K] \, c \\
&=  \frac{1}{\# \Orb(x)} \sum_{y \in \Orb(x)} \sum_{v \in M_K} \sum_{w |v }  [E_w:K_v] \, \lambda'_{D,w}(y) - [E:K] \, c \\ 
&= \sum_{v \in M_K} \sum_{w |v }  [E_w:K_v]  \frac{1}{\# \Orb(x)} \sum_{y \in \Orb(x)}  \lambda'_{D,w}(y) - [E:K] \, c .
\end{split}
\]

The various changes in the order of summation are justified by the fact that all sums are, in fact, finite sums. Using Lemma~\ref{lem:Galois_orbits}, we obtain
\[
[E : K] \, \h'_L(x)
= \sum_{v \in M_K} \sum_{w |v }  [E_w:K_v]  \frac{1}{\# \Orb(x)} \sum_{y \in \Orb(x)}   \lambda'_{D,v}(\epsilon_v(y)) - [E:K] \, c .
\]
Following \cite[\S1.3.12]{bg_heights}, we have
\[ [E:K]=\sum_{w | v} [E_w:K_v ] . \]
The required formula follows.
\end{proof}

\section{The canonical height of a divisor}\label{sec:can_ht_subvar}

We continue to work with an abelian variety $A$ defined over a field $K$, which is of transcendence degree one over an algebraically closed field $\mathds{k}$. Let $L$ be an \emph{ample} line bundle on $A$ and let $X\subseteq A$ be a closed irreducible subvariety of $A$ defined over $K$. One has the associated \emph{canonical height} $\h'_L(X) \in \rr$ of $X$ with respect to~$L$. It can be obtained as a (suitably normalized) arithmetic intersection product \cites{clt, ph, gu-hohen, bost_duke, bgs, zhsmall}. 

More precisely, equip $L$ with a rigidification at the origin and, for each $v \in M_K$ let $\|\cdot\|_{L,v}$, denote the canonical metric on $L_v^\an$ at $v$. We write  $\overline{L}=(L, \{ \|\cdot \|_{L,v} \}_{v \in M_K} )$ for the associated \emph{adelic line bundle} on $A$, in the sense of \cites{zhsmall}.
Writing $d=\dim(X)$, the canonical height of $X$ with respect to $L$ is given by the formula
\[ \h'_L(X)=\frac{1}{(d+1)\deg_L(X)} \langle \overline{L} \cdots \overline{L} | X \rangle,  \]
where the expression $\langle \overline{L} \cdots \overline{L} | X \rangle$ denotes the $(d+1)$-fold arithmetic self-intersection product of $\overline{L}$ over $X$.  For $d=0$, we re-obtain the N\'eron--Tate heights of points discussed in \S\ref{sec:can_heights_points}. There is a straightforward extension of $\h'_L$ to subvarieties of $A_{\overline{K}}$. 

\subsection{Small points and vanishing of the height} 

Assume in this section that $L$ is both \emph{ample} and \emph{symmetric}. Under these hypotheses we have $\h'_L(X) \ge 0$ (see, for example, \cite[Proposition~3.4.1]{moriwaki_arithmetic}). 

\begin{definition}
We say that a closed subvariety $X$ of $A_{\overline{K}}$  \emph{has Zariski dense small points} with respect to $L$ if, for every $\eps >0$, the set of points $x \in X(\overline{K})$ with $\h'_L(x) \leq \eps$ is Zariski dense in $X$. 
\end{definition}

By \cite[Corollary~4.4]{gubler_bog} we have the following.
\begin{thm} \label{thm:dense_small_height_vanish} Assume that $L$ is ample and symmetric. Let $X$ be an irreducible closed subvariety of $A_{\overline{K}}$. Then $X$ has Zariski dense small points with respect to $L$ if and only if $\h'_L(X)=0$.
\end{thm}

\subsection{A Mahler measure-type formula}
Let $X$ be a closed irreducible subvariety of $A$ of dimension $d$. For every place $v \in M_K$, we have a natural measure $c_1(L,\|\cdot\|_{L,v})^d \wedge \delta_{X_v^\an}$ on $A_v^\an$ called the \emph{Chambert-Loir measure} of $X$ with respect to the ample line bundle~$L$ and its canonical metric $\|\cdot\|_{L,v}$ at $v$, see \cite[\S2.4]{cl}. The Chambert-Loir measure features in a recursive formula for computing the height, as follows. Recall that an ample line bundle on an abelian variety is effective.
\begin{thm} \label{eqn:recursive_CLT}
Let $s \in H^0(A,L)$ be a nonzero global section and assume that $X$ is not contained in the divisor $\divisor_L(s)$. Let $Z \subseteq X$ be the restriction of the divisor $\divisor_L(s)$ to $X$. Then,
\begin{equation} 
\langle \overline{L} \cdots \overline{L} | X \rangle = \langle \overline{L} \cdots \overline{L} | Z \rangle - \sum_{v \in M_K} \int_{A_v^\an} \log \|s\|_{L,v} \, c_1(L,\|\cdot\|_{L,v})^d \wedge \delta_{X_v^\an}.
\end{equation}
\end{thm}
\begin{proof} See \cite[Th\'eor\`eme~4.1]{clt}.
\end{proof}
Let $g = \dim A$. Recall that on $A_v^\an$ we have a canonical measure $\mu_{H,v}$ derived from the normalized Haar measure on the canonical skeleton $\varSigma_v$ of $A_v^\an$.
\begin{thm} \label{thm:gubler_measure} The Chambert-Loir measure $c_1(L,\|\cdot \|_v)^g$ over $A$ associated to the canonical metric $\|\cdot\|_{L,v}$ on $L$ at $v$ satisfies
\[c_1(L,\|\cdot \|_{L,v})^g=\deg_L(A) \,  \mu_{H,v} . \]
\end{thm}
\begin{proof} See \cite[Corollary~7.3]{gu}.
\end{proof}
We use the above results to derive the following Mahler measure-type formula for the canonical height of an ample divisor.
\begin{thm} \label{prop:local_to_global_divisor} Let $D$ be an integral ample divisor on $A$. Write $L_0$ for the line bundle associated to $D$ and endow $L_0$ with a rigidification along the origin. Let $s \in H^0(A,L_0)$ be any nonzero global section of $L_0$ with $\divisor_{L_0}(s)=D$, and write $L$ for the symmetric ample line bundle $ L_0 \otimes [-1]^*L_0$. Then, the equality
\[ g \cdot \h'_L(D) =  \sum_{v \in M_K} \int_{A_v^\an} \log \|s\|_{L_0,v} \, \d \, \mu_{H,v} \]
holds.
\end{thm}
\begin{proof} Let $v$ be a place of $K$. By Theorem~\ref{thm:gubler_measure}, we have
\[c_1(L,\|\cdot\|_{L,v})^g=\deg_L(A)  \, \mu_{H,v} . \]
Next we apply Theorem~\ref{eqn:recursive_CLT} to the section $s \otimes [-1]^* s$ of $L$ with $X$ being $A$ itself. This gives 
\[  \langle \overline{L} \cdots \overline{L} | A\rangle  = \langle \overline{L} \cdots \overline{L}  | D + [-1]^*D \rangle
- \deg_L(A) \sum_{v \in M_K} \int_{A_v^\an} \log \| s \otimes [-1]^*s \|_{L,v} \, \d \, \mu_{H,v}  . \]
Now note that we have
\[  \log \| s \otimes [-1]^*s \|_{L,v} = \log \|s\|_{L_0,v} + \log \| [-1]^*s \|_{[-1]^*L_0,v} = 2 \log \| s \|_{L_0,v} . \]
This gives
\[\langle \overline{L} \cdots \overline{L} | A\rangle=\langle \overline{L} \cdots \overline{L}  | D + [-1]^*D \rangle-2 \cdot \deg_L(A) \sum_{v \in M_K} \int_{A_v^\an}  \log \| s \|_{L_0,v} \, \d \, \mu_{H,v}.\]
By Theorem~\ref{thm:dense_small_height_vanish} we have
\[\langle \overline{L} \cdots \overline{L} | A \rangle = 0 . \]
We also compute, using the fact that the arithmetic intersection product is $\Aut(A)$-invariant,
\[ \begin{split} \langle \overline{L} \cdots \overline{L}  | D  + [-1]^*D \rangle &=\langle \overline{L} \cdots \overline{L}  | D \rangle +  \langle \overline{L} \cdots \overline{L}  | [-1]^* D \rangle \\
&=\langle \overline{L} \cdots \overline{L}  | D \rangle +  \langle [-1]^* \overline{L} \cdots [-1]^*\overline{L}  |  D \rangle \\
& = 2 \cdot \langle \overline{L} \cdots \overline{L}  | D \rangle  \\
& = 2 \cdot \deg_L(A) \cdot g \cdot \h'_L(D) . 
\end{split} \]
The result follows.
\end{proof}

\subsection{Geometric Bogomolov conjecture for divisors}

We have the following result due to Yamaki, see \cite[Theorem~1.4]{yamaki}. We note that this result was later generalized to any irreducible closed subvarieties by Xie--Yuan in \cite[Theorem~1.1]{xie-yuan}.
\begin{thm} \label{thm:geom_bogomolov} Let $D$ be an integral divisor of $A_{\overline{K}}$. Let $L$ be an ample and symmetric line bundle on $A$. If $D$ has Zariski dense small points with respect to~$L$, then $D$ is a $\overline{K}/\mathds{k}$-special subvariety.
\end{thm}
Together with Theorem~\ref{thm:dense_small_height_vanish}, we find the following.
\begin{cor} \label{cor:height_zero_special} Let $D$ be an integral divisor of $A_{\overline{K}}$. Let $L$ be  an ample and symmetric line bundle on $A$. If $\h_L'(D)=0$, then $D$ is a $\overline{K}/\mathds{k}$-special subvariety.
\end{cor}

\section{Logarithmic equidistribution of torsion} \label{sec:log_equid}

We continue to work with an abelian variety $A$ defined over a field $K$ that is of transcendence degree one over an algebraically closed field $\mathds{k}$. We fix a place $v \in M_K$ as well as a $K$-embedding $\iota_v : \overline{K} \hookrightarrow \overline{K_v}$. We denote by  
 $\epsilon_v : A(\overline{K}) \hookrightarrow A(\overline{K_v})$ the resulting map. For $x \in A(\overline{K})$, we denote by $\Orb(x)$ the $\Aut(\overline{K}/K)$-orbit of $x$.

We have the following equidistribution result  for torsion points in $A(\overline{K})$ due to Faber and Gubler. Recall that a net $(x_n)_n$ in $A(\overline{K})$ is called \emph{generic} if it converges to the generic point of $A_{\overline{K}}$.

\begin{thm} \label{thm:equid}  Let $(x_n)_n$ be a generic net of torsion points in $A(\overline{K})$. 
Let $f \colon A_v^\an \to \rr$ be any continuous function.
Then, the equality
\[ \lim_{n} \frac{1}{\# \Orb(x_n)} \sum_{y \in \Orb(x_n)} f(\eps_v(y)) = \int_{A_v^\an} f \, \d \, \mu_{H,v}\]
holds. 
\end{thm}
\begin{proof} This is a special case of \cite[Theorem~4.1]{faber_equid} and of \cite[Theorem~1.3]{gub_equid}.
\end{proof}
The aim of this section is to  deduce from Theorem~\ref{thm:equid} the following result, which one may think of as a kind of ``logarithmic equidistribution'' property of torsion orbits. For the notion of local height function with respect to a divisor we refer to \S\ref{sec:normalized}.
\begin{thm} \label{thm:log_equid} Assume that $\mathds{k}$ is the algebraic closure of a finite field of characteristic~$p$.  Let $D$ be an effective divisor on $A$ defined over~$K$.  
Let $\lambda_{v}$ be any local height function for $D$ at $v$. Let $(x_n)_n$ be a generic net of torsion points in $A(\overline{K})\setminus|D|$ of order prime to $p$. Then, the equality
\[ \lim_{n} \frac{1}{\# \Orb(x_n)} \sum_{y \in \Orb(x_n)} \lambda_{v}(\eps_v(y)) = \int_{A_v^\an} \lambda_{v} \, \d \,\mu_{H,v} \]
holds. 
\end{thm}
Note that as $\lambda_{v}$ does not extend continuously over $|D|$, Theorem~\ref{thm:equid} can not be immediately applied to give Theorem~\ref{thm:log_equid}. 

\begin{proof}Without loss of generality, we may assume that $D$ is an integral divisor.
Let $d(\cdot,D) \colon A(\cc_v) \to \rr_{\ge 0}$ be a distance function for $D$ with respect to $v$ (see \S\ref{sec:distance}). Let $U \subseteq A$ be a Zariski open subset where $D$ is given by a local equation $z$. By Proposition \ref{voloch_distance}\,(i) we have
\[-\log d(P,D)=-\log|z(P)|_v + O(1) \, , \quad P \in (U \setminus |D|)(\cc_v) . \]
Here and in the next equation $O(1)$ denotes a bounded continuous function on $U(\cc_v)$ in the analytic topology. By \eqref{eqn:local_eqn}, we have
\[\lambda_{v}(P) = -\log|z(P)|_v + O(1)  \, , \quad P \in (U \setminus |D|)(\cc_v) . \]
These two equalities together with a compactness argument give 
\[-\log d(P,D)=\lambda_{v}(P)+O(1)  \, , \quad P \in (A \setminus |D|)(\cc_v) \, ,  \]
where $O(1)$ denotes a continuous and bounded function on $A(\cc_v)$. 

As $D$ is defined over~$K$ we have, for each $n$, that $\Orb(x_n) $ is disjoint from $|D|$. From the assumptions, it follows that $A$ may be defined over a function field $K_0$ over a finite field. By Theorem~\ref{thm:tate-voloch-app}, we may choose a $c \in \rr$ such that, for all $n$ and for all $y \in \Orb(x_n)$, we have $\lambda_{v}(y) < c$.
Since $\lambda_{v}$ is bounded on the compact set $\varSigma_v$, by making $c$ larger, we can additionally assume that $\lambda_{v}$ is bounded from above by $c$ on $\varSigma_v$.  Now set $f = \min \{ \lambda_{v} , c \}$; then $f$ is a continuous function on $A^\an_v$. By Theorem~\ref{thm:equid}, we have
\[  \lim_{n} \frac{1}{\# \Orb(x_n)} \sum_{y \in \Orb(x_n)} f(\eps_v(y)) = \int_{A^\an_v} f \, \d \, \mu_{H,v} . \]
By construction of the real number $c$, in both the expression on the left hand side and the expression on the right hand side the function $f=\min\{\lambda_{v},c\}$ can be replaced by $\lambda_{v}$, and the result follows.
\end{proof}

\section{Proof of the Main Theorem} \label{sec:proof_last_reduction_step}

We repeat the statement of the ~\mainthmrefcolor~ for convenience.
\begin{thm}[= ~\mainthmrefcolor] \label{thm:main2}
Let $K$ be a field of transcendence degree one over a finite field~$k$ of characteristic~$p$. 
Let $\overline{K}$ be an algebraic closure of $K$ and let $\overline{k} \subseteq \overline{K}$ denote the algebraic closure of $k$ in $\overline{K}$.
Let $A$ be an abelian variety over $K$.
Let $D$ be a nonzero effective divisor on $A_{\overline{K}}$.
Assume that at least one irreducible component of $D$ is not $\overline{K}/\overline{k}$-special.
Let $S$ be a finite set of places of $K$ and
let $\A$ be a model of $A$. 
Then, the set of torsion points of $A(\overline{K})$ of prime-to-$p$ order that are $S$-integral with respect to $D$ and $\A$ is not Zariski dense in $A_{\overline{K}}$.
\end{thm}
It is clear that we may replace the field $K$ by a field $K$ of transcendence degree one over the algebraically closed field~$\overline{k}$. We will have this assumption in place from now on.

Our next step is to reduce Theorem~\ref{thm:main2} to a special case. First, we note that it suffices to prove Theorem~\ref{thm:main2} for the \emph{N\'eron model} $\N$ of $A$ over $C$, and to assume that $A$ has good reduction outside $S$. Indeed, by enlarging $S$, we only make the statement stronger. Then, we enlarge $S$ in such a way that the given model $\A$ and the N\'eron model $\N$ coincide outside of $S$, \emph{and} $A$ has good reduction outside $S$.

Next, by replacing $K$ by a finite field extension we can assume all irreducible components of $D$ are defined over $K$. We can further assume that $D$ is integral, since removing irreducible components from $D$ only makes the statement stronger. 

The final reduction step is to show that we can assume that the integral divisor $D$ is \emph{ample}. Let $H$ be the identity component of the  stabilizer $\Stab_A(D)$ of $D$. Form the quotient abelian variety $A'=A/H$, and let $p\colon A\to A'$ be the projection map. The quotient $D'=D/H$ is an effective divisor of $A'$. By construction we have that $D'$ has a finite stabilizer. By \cite[Application~1, p.~60]{mum_AV}, we have that $D'$ is ample. By Lemma~\ref{lem:stab_quotient}, we have that $D'$ is not $\overline{K}/\overline{k}$-special when $D$ is not $\overline{K}/\overline{k}$-special.

Let $\Sigma$ be the set of $S$-integral torsion points of $(A \setminus |D|)(\overline{K})$ of prime-to-$p$ order. As the map $p \colon A\to A'$ is defined over $K$ we find $p(\Sigma)$ to be a set of torsion points of prime-to-$p$ order, contained in $(A' \setminus |D'|)(\overline{K})$. We observe that all elements of $p(\Sigma)$ are $S$-integral with respect to~$D'$. Indeed, if $v$ is a prime outside $S$, then by an earlier reduction step we may assume $v$ is a prime of good reduction of $A$, and hence, by \cite[Corollary~3]{serre_tate}, also of $H$ and of $A'$; and if $x$ on $A$ does not meet $ |D|  $ modulo $v$, then $p(x)$ on $A'$ does not meet $|D'|$ modulo $v$.

If we assume the truth of Theorem~\ref{thm:main2} for the ample integral non-special divisor $D'$ on the abelian variety $A'$, we have $\cl [p(\Sigma)] $ is not equal to $A'$. Since $p( \cl [\Sigma] )$ is contained in $\cl [p(\Sigma)]$, by continuity, we find that $p( \cl[\Sigma])$ is properly contained in $A'$, and hence $\cl [\Sigma]$ is properly contained in $A$. 

We have established that Theorem~\ref{thm:main2} can be reduced to the following statement.

\begin{thm} \label{thm:main_reduced} Let $K$ be a field of transcendence degree one over the algebraic closure $\overline{k}$ of a finite field $k$ of characteristic~$p$. 
Let $\overline{K}$ be an algebraic closure of $K$. 
 Let $A$ be an abelian variety over $K$. Let $D$ be an ample integral divisor on $A$ defined over $K$.  Let $S$ be a finite set of places of $K$ such that $A$ has good reduction outside $S$. Let $\N$ be the N\'eron model of $A$. Assume that $D$ is not $\overline{K}/\overline{k}$-special. Then, the set of torsion points of $A(\overline{K})$ of prime-to-$p$ order that are $S$-integral with respect to $D$ and $\mathcal{N}$ is not Zariski dense in $A_{\overline{K}}$.
\end{thm}
\begin{proof} We show the contrapositive statement. Thus, assume that the set of torsion points of $A(\overline{K})$ of order prime to $p$ that are $S$-integral with respect to $D$ and $\N$ is Zariski dense in $A_{\overline{K}}$. This means that there exists a generic net $(x_n)_n$ of torsion points of $A(\overline{K})$ of order prime to $p$ that are $S$-integral with respect to $D$ on $\N$. We are going to prove that $D$ is $\overline{K}/\overline{k}$-special.

Write $L_0$ for the line bundle associated to $D$ and write $L$ for the symmetric ample line bundle $L_0 \otimes [-1]^*L_0$.  Endow $L_0$ with a rigidification along the origin. Let $s \in H^0(A,L_0)$ be a global section such that $\divisor_{L_0}(s) = D$. 

For $x \in A(\overline{K})$ we denote by $\Orb(x)$ the $\Aut(\overline{K}/K)$-orbit of $x$. For each $v \in M_K$, we fix a $K$-embedding $\iota_v \colon \overline{K} \hookrightarrow \overline{K_v}$ and denote by   $\epsilon_v \colon A(\overline{K}) \hookrightarrow A(\overline{K_v})$ the resulting map. For each $n$ we have $\h'_{L_0}(x_n)=0$, by Proposition~\ref{prop:standard_height}\,(ii). 
From Proposition~\ref{prop:local_to_global} we deduce that for each $n$ we have
\[ \sum_{v \in M_K}  \int_{A_v^\an} \log\|s\|_{L_0,v} \, \d \, \mu_{H,v} =  \sum_{v \in M_K}   \frac{1}{\# \Orb(x_n)} \sum_{y \in \Orb(x_n)} \lambda'_{D,v}(\eps_v(y)) .  \]
The sums on both sides are finite.

Let $g=\dim(A)$. By Theorem~\ref{prop:local_to_global_divisor}, we deduce that
\[ g \cdot \h'_L(D) = \sum_{v \in M_K}   \frac{1}{\# \Orb(x_n)} \sum_{y \in \Orb(x_n)} \lambda'_{D,v}(\eps_v(y)) .  \]
By Corollary~\ref{cor:normalized_good} we have, for all $v \notin S$, for all $n$, and for all $y \in \Orb(x_n)$, that
\[ \lambda'_{D,v}(\eps_v(y))=0 \,  .\]
Indeed, any such $y$ is $S$-integral with respect to $D$ and $\N$. Thus we have
\[g\cdot \h'_L(D)=\sum_{v \in S} \frac{1}{\# \Orb(x_n)} \sum_{y \in \Orb(x_n)} \lambda'_{D,v}(\eps_v(y)).\]
The lefthand side is independent of $n$ and, hence, we might as well write
\[g \cdot \h'_L(D) = \lim_n \sum_{v \in S}   \frac{1}{\# \Orb(x_n)} \sum_{y \in \Orb(x_n)} \lambda'_{D,v}(\eps_v(y)).\]
The sum within the limit is finite, so that we may swap the limit and the sum, and we arrive at
\[g\cdot \h'_L(D)= \sum_{v \in S} \lim_n \frac{1}{\# \Orb(x_n)} \sum_{y \in \Orb(x_n)} \lambda'_{D,v}(\eps_v(y)).\]
By Theorem~\ref{thm:log_equid}, and by construction of the normalized canonical local heights $\lambda'_{D,v}$, for each $v \in S$, the limit vanishes. The assumptions of the theorem imply that $g>0$. We conclude that $\h'_L(D)=0$. By Corollary~\ref{cor:height_zero_special}, we find that $D$ is $\overline{K}/\overline{k}$-special. This finishes the proof of Theorem~\ref{thm:main_reduced}.
\end{proof}


\end{document}